\documentclass[a4paper,11pt]{article}

\usepackage{amsmath,amssymb,amsthm,dsfont,fancyhdr,color,scalerel, relsize, psfrag, graphicx, caption, mathrsfs, euscript,comment,authblk}

\usepackage{xcolor} 
\definecolor{unbleu}{rgb}{0.03, 0.15, 0.4}
\definecolor{mygreen}{rgb}{0.0,.5,0.0}
\definecolor{britishracinggreen}{rgb}{0.0, 0.26, 0.15}
\definecolor{myblue}{rgb}{0,.2,.8}
\definecolor{myotherblue}{rgb}{0,0.4,.75}
\definecolor{applegreen}{rgb}{0.55, 0.71, 0.0}
\definecolor{monrouge}{rgb}{0.8, 0.0, 0.0} 
\definecolor{cadmiumgreen}{rgb}{0.0, 0.42, 0.24}
\definecolor{black}{rgb}{0.0, 0.0, 0.0}
\definecolor{sepia}{rgb}{0.44, 0.26, 0.08}
\definecolor{teagreen}{rgb}{0.82, 0.94, 0.75}
\definecolor{yellow-green}{rgb}{0.6, 0.8, 0.2}
\definecolor{azure(colorwheel)}{rgb}{0.0, 0.5, 1.0}
\definecolor{awesome}{rgb}{1.0, 0.13, 0.32}
\definecolor{cadmiumyellow}{rgb}{1.0, 0.96, 0.0}
\definecolor{royalblue(traditional)}{rgb}{0.93, 0.57, 0.13}
\definecolor{green-yellow}{rgb}{0.68, 1.0, 0.18}
\definecolor{huntergreen}{rgb}{0.21, 0.37, 0.23}

\usepackage{hyperref}
\hypersetup{linkcolor=huntergreen,citecolor=royalblue(traditional), urlcolor=azure(colorwheel), colorlinks=true} 

\newtheorem{thm}{Theorem}[section]
\newtheorem{lem}[thm]{Lemma}
\newtheorem{prop}[thm]{Proposition}
\newtheorem*{thm*}{Theorem}

\newtheorem{corollary}[thm]{Corollary}
\newtheorem*{thma}{Theorem A}
\newtheorem*{thmb}{Theorem B}
\newtheorem*{thmbprime}{Theorem B'}

\theoremstyle{definition}
\newtheorem{defn}[thm]{Definition}
\newtheorem{rem}[thm]{Remark}

\newcommand{\R}{\mathbb R}
\newcommand{\Z}{\mathbb Z}
\newcommand{\N}{\mathbb N}
\newcommand{\cL}{{\mathscr L}}
\newcommand{\cQ}{\mathcal Q}
\newcommand{\cF}{\mathcal F}
\newcommand{\cP}{\mathcal P}
\newcommand{\cA}{\mathcal A}
\newcommand{\cB}{\mathcal B}
\newcommand{\cT}{\mathcal T}
\newcommand{\bi}{{\boldsymbol{i}}}
\newcommand{\bn}{{\boldsymbol{n}}}
\newcommand{\bm}{{\boldsymbol{m}}}

\newcommand{\bq}{{\boldsymbol{q}}}
\newcommand{\bj}{{\boldsymbol{j}}}
\newcommand{\bk}{{\boldsymbol{k}}}
\newcommand{\bv}{{\boldsymbol{v}}}
\newcommand{\bu}{{\boldsymbol{u}}}
\newcommand{\comp}{\mathrm{c}} 
\newcommand{\Zd}{{\mathbb Z^d}}
\newcommand{\ES}{\mathrm{ES}}
\newcommand{\MM}{\mathrm{MM}}
\newcommand{\ZTA}{\mathrm{ZTA}}
\newcommand{\sint}{\scaleto{\int}{22pt}} 
\newcommand{\betac}{\beta_{\mathrm{c}}}
\newcommand{\ergbox}[1]{\Lambda^{\!\scaleobj{.8}{+}}_{#1}}
\newcommand{\statbox}[1]{\Lambda_{#1}}
\newcommand{\dbar}{ \overline{\mathrm{d}}}

\newcommand{\dd}{{\,\mathrm d}} 
\newcommand{\widesim}{\scaleobj{1.3}{\sim}}
\newcommand{\topo}{\scaleobj{0.8}{\rm top}}

\newcommand{\stwo}{{\scaleto{2}{3pt}}}

\DeclareMathAlphabet{\dutchcal}{U}{dutchcal}{m}{n}
\SetMathAlphabet{\dutchcal}{bold}{U}{dutchcal}{b}{n}
\DeclareMathAlphabet{\dutchbcal} {U}{dutchcal}{b}{n}
\DeclareMathOperator{\ent}{\scaleobj{1.15}{\dutchcal{h}}} 
\newcommand{\Ent}{H}

\let\phi=\varphi

\DeclareMathOperator{\var}{var}
\DeclareMathOperator{\card}{card}
\DeclareMathOperator{\lexmin}{lex\,min}
\DeclareMathOperator{\dist}{dist}
\DeclareMathOperator{\e}{\mathrm{e}}
\DeclareMathOperator{\1}{\mathds{1}}

\makeatletter
\newcommand{\uset}[3][0ex]{%
  \mathrel{\mathop{#2}\limits_{
    \vbox to#1{\kern-6\ex@
    \hbox{$\scriptstyle#3$}\vss}}}}
\makeatother

\begin{document}

\title{Freezing Phase Transitions for Lattice Systems and Higher-Dimensional Subshifts}

\author[1]{J.-R. Chazottes
\thanks{Email: \texttt{jeanrene@cpht.polytechnique.fr}}
}

\author[2]{T. Kucherenko
\thanks{Email: \texttt{tkucherenko@ccny.cuny.edu}}
}

\author[3]{A. Quas
\thanks{Email: \texttt{aquas@uvic.ca}}
}

\affil[1]{{\small CPHT, CNRS, \'Ecole polytechnique, Institut Polytechnique de Paris, 91120 Palaiseau, France}}
\affil[2]{{\small Department of Mathematics, The City College of New York, New York, NY, 10031, USA}}
\affil[3]{{\small Department of Mathematics \& Statistics, University of Victoria, Victoria BC, Canada v8W 2Y2}}

\date{Dated: \today}

\maketitle

\begin{abstract} 
Let $X = \cA^{\Z^d}$, where $d \geq 1$ and $\cA$ is a finite set, equipped with the action of the shift map.
For a given continuous potential $\phi: \cA^{\Z^d} \to \R$ and $\beta>0$ (``inverse temperature''), there exists a (nonempty) set of equilibrium states $\ES(\beta\phi)$. The potential $\phi$ is said to exhibit a ``freezing phase transition'' if $\ES(\beta\phi) = \ES(\beta'\phi)$ for all $\beta, \beta' > \betac$, while $\ES(\beta\phi) \neq \ES(\beta'\phi)$ for any $\beta < \betac < \beta'$, where $\betac\in (0,\infty)$ is a critical inverse temperature depending on $\phi$. 

In this paper, given any proper subshift $X_0$ of $X$, we explicitly construct
a continuous potential $\phi: X \to \R$ for which there exists 
$\betac \in (0,\infty)$ such that $\ES(\beta\phi)$ coincides with the 
set of measures of maximal entropy on $X_0$ for all $\beta > \betac$, whereas 
for all $\beta < \betac$, $\mu(X_0)=0$ for all $\mu\in\ES(\beta\phi)$. 
This phenomenon was previously studied only for $d = 1$ in the context of dynamical 
systems and for restricted classes of subshifts, with significant motivation stemming 
from quasicrystal models.  
Additionally, we prove that under a natural summability condition -- satisfied, 
for instance, by finite-range potentials or exponentially decaying potentials -- 
freezing phase transitions are impossible.  

\bigskip


\end{abstract}

\tableofcontents

\section{Introduction}

The connections between statistical physics and dynamical systems have long been 
fruitful, with ideas initially flowing from physics into the study of dynamical systems. 
The foundational developments in ergodic theory and chaotic dynamics were deeply 
influenced by statistical mechanics, particularly through the notion of Gibbs states. 
These measures, introduced by the pioneering work of Boltzmann and Gibbs, became a 
crucial tool in understanding long-term statistical behavior in dynamical systems, 
notably in the study of hyperbolic systems and Anosov flows. The development of 
thermodynamic formalism, initiated by Sinai, Ruelle, and Bowen, brought a statistical 
physics perspective to dynamical systems, using concepts such as entropy, pressure, 
Gibbs and equilibrium states to describe the statistical properties of chaotic 
systems \cite{CK,Chernov}.

However, in recent years, this flow of ideas has started to reverse in certain contexts. 
Techniques originating in ergodic theory and symbolic dynamics -- historically developed 
for understanding dynamical systems -- are now being applied back to lattice models in 
statistical physics.
Lattice models have played a pivotal role in the study of phase transitions, providing a 
mathematical framework to understand how systems change between distinct states of 
matter (e.g., from a disordered phase to an ordered phase). These models are defined on 
a discrete lattice ({\em e.g.}, $\Z^d$) and describe interactions between variables 
({\em e.g.}, spins, particles, or other degrees of freedom) located at the lattice 
sites; see for instance \cite{FV,Georgii}. 

Recent studies on low-temperature limits and ground states of lattice models have 
increasingly leveraged techniques from symbolic dynamics \cite{BBDT, CGU, CH, CS, CRL, GST}. In this paper, we focus on 
freezing phase transitions, where the set of equilibrium states can collapse at low 
temperatures -- not just at zero temperature. While our results on freezing phase 
transitions are abstract, they were motivated by understanding how equilibrium states 
collapse in a way that they become concentrated on subshifts without any periodic 
configuration, which serve as models for quasicrystals in the context of lattice systems. 

Despite the close connections between statistical physics and dynamical systems through 
thermodynamic formalism \cite{Ruelle}, these two fields remain largely distinct communities. This 
separation arises not only from differences in the types of questions they address but 
also from fundamental differences in their core objects and frameworks. In statistical 
physics, the primary object is an interaction, sometimes called an ``interaction potential'', which is a family of functions 
indexed by finite subsets of $\mathbb{Z}^d$. This object describes how components of a 
system interact locally and provides a natural, flexible foundation for building models 
from the ground up. In contrast, in dynamical systems, the central object is a function 
on the phase space -- often called a potential as well.

Schematically (we will revisit this point in more detail later), interactions 
are used to define Gibbs states as well as equilibrium states\footnote{\, In the context of
statistical physics, one uses the term ``state'' to mean a probability measure, especially for
Gibbs and equilibrium measures which describe the states of the system.}. The construction of 
these states relies on an averaged version of the interactions over finite subsets of $\Z^d$
containing the origin. These averaged functions correspond to 
the potentials used in dynamical systems, where the natural objects are equilibrium 
states rather than Gibbs states (notably because the former can be defined in a very general setting). The complication arises from the fact that the 
mapping that transforms an interaction into a potential lacks good properties. 
As a result, it is generally not 
possible to automatically translate results obtained in the context of dynamical systems 
to apply them directly to interactions. This effect will show up in this 
article: although our results are stated in the spirit of dynamical systems, we will 
show how to adapt and translate them into the framework of statistical physics.

\subsection{Equilibrium states, temperature going to \texorpdfstring{$0$}{} and maximizing measures}
Consider the lattice $\Z^d$ and a space of configurations $X$ of the form
$\cA^{\Zd}$ where $\cA$ is a finite set (often called an alphabet in the context of 
symbolic dynamics). On $X$ we consider the action of the shift map. 
A potential $\phi$ is a real-valued function on $X$.
Given $\beta\in \R$, the equilibrium states for $\beta\phi$ are the shift-invariant probability measures  
which maximize the functional 
\[
\nu\mapsto \ent(\nu) +\beta \int \phi\dd\nu,
\]
where $\ent(\nu)$ is the entropy of $\nu$. This is the variational principle.
We also define the so-called pressure function 
\begin{equation}
\label{def:pressure-function}
\beta\mapsto p_\phi(\beta):=\sup_{\nu}\Big\{\ent(\nu)+\beta \sint \phi\dd\nu\Big\},
\end{equation}
where the supremum is taken over shift-invariant (probability)\footnote{\, Unless otherwise specified, all measures in this paper are probability measures.}
measures. 
The pressure function is convex and in particular continuous.
(For this material, we refer to \cite{Keller}.)

The parameter $\beta$ is interpreted as the inverse temperature in the context
of statistical physics so, physically, $\beta\geq 0$.
For $\beta=0$ (infinite temperature), there is only one equilibrium state 
which is the measure of maximal entropy on $\cA^{\Zd}$, that is, the 
product measure whose entropy is equal $\log |\cA|$. This is the most disorderly phase possible, where $\phi$ plays no role.

The other extreme is the limit $\beta\to+\infty$ (the zero-temperature limit) which is associated with ground states and is considerably more subtle. It has only begun to be the subject of systematic research in the last 
fifteen years, and revealed a number of surprises, see \cite{BBDT, CGU, CH, CS, CRL, GST} and references therein. 
A maximizing measure for $\phi$ is a shift-invariant measure 
$\mu$ such that $\int \phi \dd\mu\geq \int \phi \dd\nu$ for all shift-invariant measures $\nu$.
The term ``maximizing measure'' originates from the field of ergodic optimization. In statistical physics, one would instead use the term ``ground state''. However, due to the lack of consensus on the precise definition of ground states, we adopt the terminology from dynamical systems here.
For a comprehensive survey on ergodic optimization in the context of a continuous map acting on a compact metric space --
corresponding to a $\Z$-action when the map is invertible -- we refer the reader to \cite{Jenkinson}. While some general results 
extend naturally to $\Z^d$-actions, others do not.
If $\phi$ is not cohomologous to a constant and satisfies a natural uniform summability condition,
the support of any of its maximizing measures is contained in a proper 
subshift of $\cA^{\Zd}$, that is, a closed and shift-invariant subset of $\cA^{\Zd}$.
This subshift can be characterized in terms of maximizing configurations which are configurations 
whose energy cannot be increased by finite modifications \cite{GT}, 
so we call it the maximizing subshift for $\phi$. 

Now, if we take an arbitrary sequence $(\beta_n)_{n\geq 1}$ of inverse temperatures such that 
$\beta_n\to+\infty$, and consider a sequence $(\mu_n)_{n\geq 1}$ where, 
for each $n$, $\mu_n$ is an  equilibrium state for $\beta_n\phi$, then any accumulation 
point of  $(\mu_n)_{n\geq 1}$ is a maximizing measure and it maximizes entropy among 
the possible maximizing measures. 
There are two special cases where the zero-temperature limit 
does exist, namely if the maximizing subshift is uniquely ergodic (i.e., carries a single shift-invariant measure), 
or if it has a unique measure of maximal entropy. The existence of the zero-temperature limit 
turns out to depend both on the regularity of $\phi$ and $d$ (the dimension of the lattice).
For $d=1$, the limit always exists for locally constant potentials (corresponding to 
finite-range interactions), but may fail to exist for Lipschitz potentials 
(i.e., exponentially decaying pair interactions). For $d\geq 2$,  the limit generally fails to exist even for locally constant potentials.

The pressure function, in any dimension and for any continuous $\phi$, exhibits a slant asymptote, that is, 
\begin{equation}\label{slant-asymptote}
\lim_{\beta\to+\infty} \big(\, p_\phi(\beta)-(s_\phi\beta +\ent_\phi)\big)=0, 
\end{equation}
where 
\begin{equation}\label{def-a-b}
s_\phi=\sup_\nu \int \phi\dd\nu\quad\mathrm{and}\quad
\ent_\phi=\sup\Big\{\ent(\eta):\sint \phi\dd\eta=s_\phi\Big\},
\end{equation}
and the supremum defining $s_\phi$ is taken over shift-invariant measures.
(For completeness, we provide a proof of this fact in Section \ref{appendix-folklore}.)

\subsection{Subshifts and quasicrystals}
We can approach the problem in reverse: starting with a subshift that models either a 
crystal or a quasicrystal, we can investigate how to construct an interaction for which this subshift will serve as the maximizing one. Here a crystal is a 
subshift consisting of finitely many configurations, each being a shift of another. 
A quasicrystal, by contrast, is a uniquely ergodic, zero-entropy subshift with 
infinitely many configurations, and thus no periodic  configurations.
We emphasize that there is no universally accepted definition of a quasicrystal, even within the framework of lattice models.
One could for instance refine the previous definition by additionally requiring the subshift to be minimal in the ergodic-theoretic sense, which informally means that all configurations share the same patterns. On the other hand, one could simply require a subshift that contains uncountably many configurations, with none of them being periodic.

There is a striking difference between $d=1$ and $d\geq 2$. To model a quasicrystal for $d\geq 2$, 
one can consider subshifts of finite type without periodic configurations, which correspond to 
aperiodic Wang tilings (we will mention examples and references below). 
For $d=1$, however, subshifts of finite type are either reduced to finitely many 
periodic configurations (crystal), or are of positive topological entropy which is 
interpreted as chaos/disorder.
Consequently, in dimension 1, natural zero-entropy uniquely ergodic subshifts such 
as the Thue-Morse substitution have been commonly used, and infinite-range
interactions are required \cite{GMRE}.
For further discussions and background, we refer the reader to \cite{vE,vEM}.

\subsection{Freezing at non-zero temperature.}
For $0 < \beta < \infty$, the variational principle described above indicates that equilibrium states exhibit a balance between ``order'' and ``disorder'', with the potential $\phi$ acting as the driving force behind ordering. However, phase transitions can still occur. 
From this perspective, we can reconsider the previous discussion by examining the behavior at nonzero temperature and addressing the following
questions:
\begin{quote}
Given a subshift modeling, for example, a quasicrystal, can we construct a function $\phi$ whose equilibrium states coincide, for all $\beta>\betac$, with the unique invariant measure supported by this subshift? Is it possible that for $\beta < \beta_c$, none of the equilibrium states for $\phi$ are supported on the given subshift?
\end{quote}
For a more general subshift, one might wish for the equilibrium states to correspond, for all $\beta>\betac$, to the measures of 
maximal entropy associated with this subshift.

If this scenario occurs, the phase transition corresponds to an abrupt contraction of the support of the equilibrium states, which become concentrated on a significantly smaller set of maximizing configurations for the given potential.
This type of phase transition differs from the 
classical phase transitions typically encountered in statistical physics 
\cite{FV,Georgii}.

In this article, given any proper subshift $X_0$ of $X=\cA^{\Zd}$, we are able to construct a continuous potential, with controlled modulus of continuity, that ``freezes'' on $X_0$. Specifically, for $\beta > \betac$, the equilibrium states coincide with the measures of maximal entropy supported on $X_0$, whereas for $\beta < \betac$, none of the equilibrium states for the given potential are supported on $X_0$. 
Additionally, we establish a no-go theorem, demonstrating that potentials satisfying a classical summability condition cannot exhibit such phase transitions. 

\subsection{Statistical physics}
Recall that $X=\cA^{\Zd}$. The primary objects here are interactions. 
An interaction $\Phi$ is a family of functions $(\Phi_\Lambda)_{\Lambda\Subset \Zd}$, 
where for each (finite) subset $\Lambda$ of $\Zd$, $\Phi_\Lambda:X\to\R$ depends
only on the configurations restricted to $\Lambda$. 
Interactions are the building blocks of Gibbs states that we will not describe
here (we refer to \cite{FV,Georgii,Ruelle}).
For a shift-invariant interaction, one can construct a function $\phi:\cA^{\Zd}\to\R$ 
given by 
\[
\phi=\sum_{\substack{\Lambda\Subset \Z^d \\ 0\in \Lambda}}\frac{\Phi_\Lambda}{|\Lambda|},
\]
where $A\Subset B$ means $A$ is a finite subset of $B$.
This function appears in the definition of equilibrium states (variational principle) \cite{Georgii,Ruelle}.
It is commonly referred to as a potential in dynamical systems theory, and is
continuous if $\sum_{\Lambda\Subset \Z^d,\, 0\in \Lambda} |\Lambda|^{-1}\|\Phi_\Lambda\|_\infty<+\infty$.

A fundamental question  is to understand the structure of the set of Gibbs and 
equilibrium states for the interaction 
$\beta \Phi$ where $\beta\geq 0$. Let us emphasize 
that equilibrium states coincide with shift-invariant Gibbs states
for sufficiently  regular interactions, the natural condition being
absolute summability, that is, 
$\sum_{\Lambda\Subset \Z^d,\, 0\in \Lambda} \|\Phi_\Lambda\|_\infty<+\infty$; see \cite[Theorem 15.30, p. 322]{Georgii} or \cite[Theorem 4.2, p. 58]{Ruelle}.
In the absence of uniqueness, there exist examples of non-shift-invariant Gibbs states, which therefore cannot be equilibrium states, as well as equilibrium states that are not Gibbs states. However, when the Gibbs state is unique, it must be shift-invariant and coincides with the unique equilibrium state.

Notice that sign conventions used in dynamical systems are different from those used in statistical physics, basically because physicists prefer to minimize energy. 

\section{Main results} \label{sec:main-results}

\subsection{Freezing on any proper subshift and a no-go result}

We first define what we precisely mean by a {\em freezing phase transition}. (We will mention several concepts that will be precisely defined in the next section.) 
Let $X=\cA^{\Zd}$ be a full $d$-dimensional shift on a finite alphabet $\cA$.
For $\beta\in\R$ and a continuous potential $\phi:X\to\R$, we define the following three subsets of shift-invariant measures. The set of equilibrium states for $\beta\phi$ is
\[
\ES(\beta\phi)=\Big\{\mu\;\text{shift-invariant}:\; \ent(\mu)+\beta\sint \phi\dd\mu=p_\phi(\beta)\Big\},
\]
where $p_\phi(\beta)$ is defined in \eqref{def:pressure-function}. Then, the set of maximizing measures for
$\phi$ is  
\[
\MM(\phi)=\Big\{\mu\;\text{shift-invariant}:\sint \phi\dd\mu=s_\phi\Big\},
\]
where $s_\phi$ is defined in \eqref{def-a-b}. Finally, define the set of zero-temperature accumulation points as
\[
\ZTA(\phi)=\big\{\mu\colon\exists \beta_n\to\infty,\ \mu_n\in \ES(\beta_n\phi)\text{ with } \mu_n\leadsto\mu\big\},
\]
where $\mu_n\leadsto\mu$ means that $\mu_n$ converges to $\mu$ for the vague topology. A well-known folk theorem states that $\ZTA(\phi) \subseteq \MM(\phi)$, with the measures in $\ZTA(\phi)$  being among those that maximize entropy among all maximizing measures (see Theorem \ref{thm:folklore} in Section
\ref{appendix-folklore}).

\begin{defn}[Freezing phase transition for a potential]\label{def:FPT}
Let $\phi:X\to\R$ be a continuous function.
We will say that $\phi$ has a freezing phase transition at $\betac$ if 
$\ES(\beta\phi)=\ES(\beta'\phi)$ for all $\beta,\beta'>\betac$,
while $\ES(\beta\phi)\ne \ES(\beta'\phi)$ for any $\beta<\betac<\beta'$.
\end{defn}
\begin{rem}\label{rem:ESdisjoint}
By the following proposition, which characterizes freezing phase transitions, it follows that imposing $\ES(\beta\phi) \ne \ES(\beta'\phi)$ for any $\beta < \beta_c < \beta'$  necessarily implies that $\ES(\beta\phi) \cap \ES(\beta'\phi) = \emptyset$ for all
$\beta < \beta_c < \beta'$. Also, if 
$\ES(\beta\phi)=\ES(\beta'\phi)$ for all $\beta,\beta'>\betac$, then $\ES(\beta\phi) = \ZTA(\phi)$ for all $\beta > \beta_c$.
However, we prefer to keep the definition as it is.
\end{rem}

The previous definition can be rephrased in terms of the pressure function's affine behavior to the right of $\betac$, as follows (see \eqref{def:pressure-function} for the definition).
The proof is given in Section \ref{appendix-folklore}.

\begin{prop}\label{prop:freezing=slant-asymptote}
Let $X=\cA^{\Zd}$ be a full $d$-dimensional shift on a finite alphabet $\cA$, $\phi:X\to\R$ be a continuous function, and $\betac\in\R$.  
Then, $\phi$ has a freezing phase transition at $\betac$ if and only if the pressure function is affine on $[\,\betac,\infty)$ and not on any larger interval. Moreover, in this case $p_\phi(\beta)=s_\phi \beta+\ent_\phi$, where $s_\phi,\ent_\phi$ are defined in \eqref{def-a-b}.
\end{prop}

According to the above proposition, if $\phi$ has a freezing phase transition, the pressure function $p_\phi(\beta)$ would exhibit a graph similar to the one shown in Figure \ref{fig:fpt}. 

\begin{figure}[htb!]
\centering
\includegraphics[scale=.25]{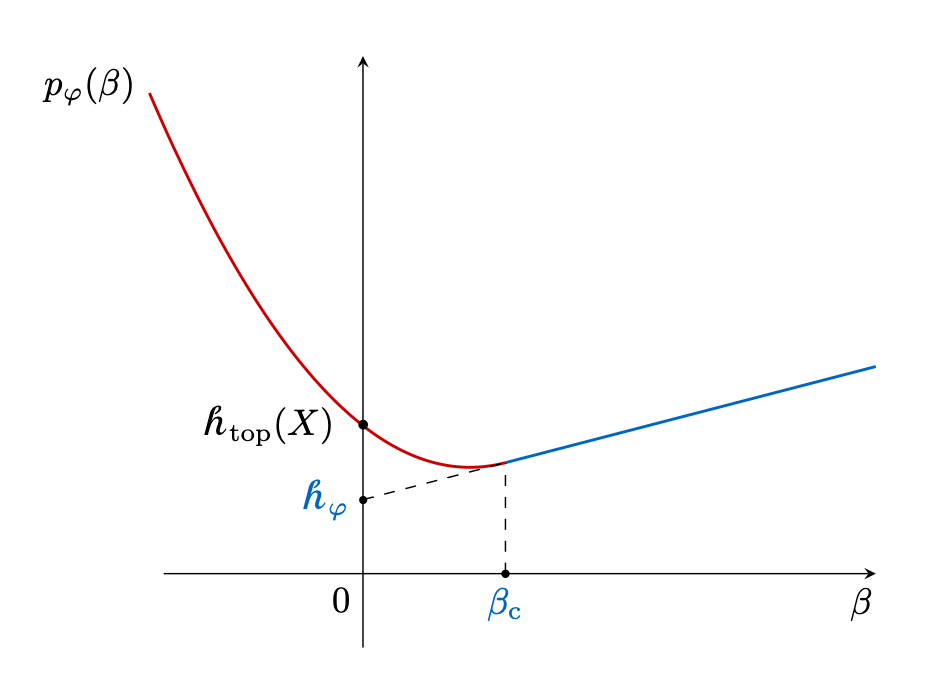}
\caption{Freezing phase transition at $\betac$.}
\label{fig:fpt}
\end{figure}
 
We now address our main question: How can we construct a continuous function $\phi$ that {\em freezes on a given subshift} $X_0$ of $X$? We make the following definition.

\begin{defn}[Freezing on a subshift]
Let $X_0$ be a proper subshift of $X$. 
A continuous potential $\phi\colon X \to \R$ is said to \emph{freeze on $X_0$} if it
has a freezing phase transition at some parameter $\beta_c$ and $X_0$ is the smallest subshift (with respect to the inclusion) which contains the supports of all measures in $\ZTA(\phi)$.
\end{defn}

Our main result is the following.

\begin{thma}\label{thm:A}
For any proper subshift $X_0$ of a full $d$-dimensional shift $X$ on a finite alphabet,
there is an explicit potential $\phi$ with a freezing phase transition at some $\beta_c>0$
such that for $\beta>\beta_c$, $\ES(\beta\phi)$ is the set of measures of maximal entropy on $X_0$
(and coincides with $\ZTA(\phi)$), while for $\beta<\beta_c$, $\mu(X_0)=0$ for all $\mu\in\ES(\beta\phi)$.
In particular, if the smallest subshift containing the supports of all measures of
maximal entropy on $X_0$ is $X_0$ itself, then $\phi$ freezes on $X_0$ at $\beta_c$.
\end{thma}

A more detailed statement is given in Theorem \ref{thm:existence} below. In particular, the potential 
takes countably-many values and its variation is bounded above by a term arising from the rate of convergence in the definition of topological entropy \eqref{eq:top_entropy} and an error term of order $\log\log n/n^d$.
In  particular, when the subshift $X_0$ is a full shift with alphabet $\cA'\subsetneq\cA$, 
$\var_n\phi=O\big(\frac{\log\log n}{n^d}\big)$. Another particular case is when $X_0$
is a constant-box substitution shift, in which case $\var_n\phi=O\big(\frac{\log n}{n^d}\big)$
(see Corollary \ref{cor:entropy_gap}). 

However, as our next theorem shows, a freezing phase transition is impossible for potentials $\phi$ satisfying $\var_n \phi \leq \frac{1}{n^{d+\epsilon}}$ for some $\epsilon > 0$ and all $n \geq 1$.

\begin{thmb}\label{thm:full_support}
Let $X$ be a full $d$-dimensional shift on a finite alphabet and $\phi:X\to \R$ be a 
continuous potential. If $\sum_{n\geq 1} n^{d-1}\var_n\phi<\infty$ then $\phi$ cannot have a freezing phase transition.
\end{thmb} 


\subsection{Some remarks about Theorem A}
One may wonder about what happens for $\beta=\betac$.
As demonstrated already by the classical example of Hofbauer (see below), $\ES(\beta_c\phi)$ may or 
may not coincide with $\ZTA(\phi)$, and the pressure function may or may not be differentiable at $\betac$.

Mathematically, one can consider negative inverse temperatures.
Associated with the limit $\beta \to -\infty$, one can consider minimizing measures. The corresponding freezing question relates to whether it is possible to reach the associated slant asymptote, which would occur to the left of a certain inverse temperature $\beta'_c$ (which can be $\neq\betac$). Clearly, this is equivalent to analyzing the problem for $\beta'(-\phi)$ with $\beta'\to\infty$ since $\beta\phi=(-\beta)(-\phi)$.

The last remark is related to the measure of maximal entropy on $X$.
If $\phi=u\circ \sigma-u+c$ where $u:X\to\R$ is continuous, $c\in\R$, and $\sigma:X\to X$ is the shift map, 
then $p_\phi(\beta)=c\beta + \log|\cA|$. Moreover, these potentials correspond to the 
same equilibrium state which is the measure of maximal entropy of $X$ whose value is
$\log|\cA|=\ent_{\topo}(X)$, the topological entropy of $X$ (see next section for more details).
All shift-invariant measures are maximizing measures for this $\phi$. 

\subsection{Intermezzo on statistical physics}

One may naturally ask what happens when starting from an interaction
$\Phi=(\Phi_{\Lambda})_{\Lambda\Subset \Z^d}$. Unfortunately, the correspondence between $\Phi$ 
and $\phi$ is rather intricate in general. For instance, the mapping $\Phi \mapsto \varphi$ is 
not injective, and there is no general method to construct preimages of $\phi$; 
see \cite[pp. 34-35]{Ruelle}). Hence, some 
information is inevitably lost when comparing the regularity of $\Phi$ with that of $\phi$.

For finite-range interactions -- which encompass most models in dimension
$d \geq 2$ -- the situation is straightforward, as this implies that $\var_n \phi$ vanishes for 
all sufficiently large $n$. Hence, Theorem B implies that finite-range 
interactions cannot have a freezing phase transition. But of course the condition
$\sum_{n\geq 1} n^{d-1}\var_n\phi<\infty$
can be satisfied by long-range interactions, provided they decay fast enough.
(If there exists a finite $R>0$ such that $\Phi_\Lambda=0$ for all
sets $\Lambda\Subset\Z^d$ whose diameter is larger than $R$, the potential is said to have
finite range. Otherwise, it has infinite range.)
The issue is that one cannot directly derive the natural assumption on $\Phi$ from that on
$\phi$, which separates interactions based on whether they can induce a freezing phase 
transition or not. Therefore, we state a theorem specifically for interactions.

\begin{thmbprime}\label{thm:full_support_Gibbs}
Let $X$ be a full $d$-dimensional shift on a finite alphabet and $\Phi=(\Phi_{\Lambda})_{\Lambda\Subset \Z^d}$ be an interaction.
If $\sum_{\Lambda\Subset \Z^d,\,\Lambda\ni 0}\|\Phi_\Lambda\|_\infty<\infty$ then $\Phi$ cannot have a freezing phase transition.
\end{thmbprime} 

The summability condition in the previous statement is satisfactory because it holds precisely for interactions belonging to the classical space of interactions. We provide more information
in Section \ref{sec:proof-theorem-B'}.
We do not claim that Theorem B' is entirely new, but we were unable to find a suitable reference.

We end this section with the following remark regarding the analog of Theorem A in terms of interactions.
\begin{rem}
Theorem \ref{thm:existence} below provides a more precise formulation of Theorem A by explicitly expressing the potentials that exhibit a freezing phase transition in terms of the target subshift $X_0$. For each such potential, we derive a formula for the corresponding interaction, which, as expected, is not summable in the sense of Theorem B' (see Remark \ref{rem:freezing-Phi} for further details).
\end{rem}

\subsection{Previously known results}

Previous studies on freezing phase transitions have focused exclusively on the case $d=1$ and specific examples. The constructions 
were carried out in one-sided settings, i.e., $X=\cA^\N$, and the advantage is two-fold. First, Ruelle's Perron-Frobenius operator 
can be used to evaluate the pressure. Second, a potential defined on a one-sided shift can be readily extended to the corresponding two-sided shift without altering the pressure function. Hence, the results concerning phase transitions 
immediately carry over to the two-sided settings. We note that the converse is not true. Walters \cite{W2} provides necessary and sufficient 
conditions for a continuous two-sided potential to be cohomologous to a one-sided potential, and these conditions cannot be 
satisfied by any potential with a freezing phase transition.  

In the context of shift dynamical systems, the first example of a freezing phase
transition is due to Hofbauer \cite{Hofbauer}, though his motivation was to exhibit a continuous 
potential $\phi$ with two distinct (ergodic) equilibrium states (there is no inverse temperature in his paper). He works with
the one-sided full shift $\{0,1\}^{\N}$ and defines the potential $\phi$ by 
$\phi(0^\infty)=0$ and $\phi(x)=a_n$ if $\dist(x,0^\infty)=2^{-n}$, where $(a_n)_{n\geq 0}$ 
is a sequence of negative real numbers such that $\sum_{n\geq 0} \e^{a_0+\cdots+a_n}=1$.
If $\sum_{n \geq 0} (n+1)\e^{a_0+\cdots+a_n} < \infty$, there are exactly two ergodic equilibrium states. Conversely, if $\sum_{n \geq 0} (n+1)\e^{a_0+\cdots+a_n} = \infty$, there is a unique equilibrium state.
Reformulated in our context, he proves that this potential has a freezing phase transition at $\betac=1$ with
$X_0=\{0^\infty\}$, hence $s_\phi=\ent_\phi=0$ and $\ES(\beta\phi)=\MM(\phi)=\ZTA(\phi)=\{\delta_{0^\infty}\}$ for 
all $\beta>1$. In these examples, prior to the freezing phase transition, the unique equilibrium state is
fully supported, whereas after the phase transition, the unique equilibrium state is supported on a single point.
When the previous series is finite, the pressure function is not differentiable at $\betac$,
while it is differentiable when it is infinite.
Hofbauer was inspired by the Fisher-Felderhof droplet model \cite{FF}, of 
which it is a crude simplification.
We refer to \cite{Ledrappier} for a shorter, alternative proof of the nonuniqueness part of Hofbauer's results.
A systematic study of a family of Hofbauer potentials exhibiting a phase transition
at some value $\betac$ was done by Lopes \cite{Lo3}.

Another example in the same spirit as Hofbauer's example is given in \cite{Olivier} to exhibit 
non-analyticity of multifractal spectra for pointwise dimension. The author considers
$X=\{0,1,2\}^{\N}$ and, reformulated in our context, he has $X_0=\{0,1\}^{\N}$ and $\betac=1$. A general case where $X=\cA^\N$ for a finite alphabet $\cA$ and $X_0\subset X$ is any subshift of finite type is studied in \cite{KTh1}, where an explicit construction of a potential with a freezing phase transition on $X_0$ is carried out. 

In the context of one-dimensional, one-sided shifts, with potentials of the same form as the ones 
used in this paper, references \cite{BL1, BL2} assert that they have established a freezing phase 
transition with $X_0$ defined by the subshift generated by either the Thue-Morse or Fibonacci 
substitutions. Unfortunately, the proofs presented in these articles are flawed.  
Work in progress \cite{BCHL} establishes a freezing phase transition in the Thue-Morse subshift, 
building on the techniques of \cite{BL1}, but requiring significant combinatorial refinements. 
In particular, since these papers use Ruelle's operator, this approach cannot be extended to 
higher dimensions.

Recent work \cite{BKL} broadly extends \cite{BL1, BL2, BCHL} by proving existence of freezing 
phase transitions for all zero-entropy (one-sided and two-sided) subshifts, instead of 
very specific ones. 
However, the proof is not constructive, since the potentials are ultimately obtained 
through the Hahn-Banach theorem. 
Recently, leveraging the same underlying techniques -- which are inherently non-constructive -- it was demonstrated in 
\cite{Hedges}, within a context encompassing ours, that for any nonempty collection of ergodic measures with a constant entropy 
function, there exists a continuous $\phi$ such that, for all $\beta \geq 1$, its closed convex hull coincides with
$\ES(\beta\phi)$.

There is an abstract result in \cite{KTh} that guarantees the existence of $\phi$ leading to a pressure
function as in Figure \ref{fig:fpt}. Precisely, given a compact 
topological dynamical system $(X,T)$ with positive entropy and upper semi-continuous entropy map, and any closed invariant subset 
$X_0\subset X$ with positive topological entropy, there exists a continuous potential $\phi: X\to\R$ satisfying (\ref{prop_one}), (\ref{prop_two}) and (\ref{prop_three}) (see Lemma \ref{lem:special-class} below). 
The proof is based on the fact that equilibrium states can be viewed as tangent 
functionals to the pressure and relies on a theorem by Israel about approximations of shift-
invariant measures by equilibrium states; see \cite{Israel-book} or \cite[Proposition 17.7, p. 341 ]{Georgii}.
It does not provide any information on the type of potentials 
which would generate a freezing phase transition. 

To conclude this section, we observe that in Hofbauer's example, as well as in \cite{BCHL}, there exists a unique equilibrium state with full support for all $0 \leq \beta < \beta_c$. In the construction provided in the proof of Theorem A (detailed in Theorem~\ref{thm:existence}), it remains unclear whether fully supported measures $\mu \in \ES(\beta\phi)$ exist for all $0 \leq \beta < \beta_c$. We leave this as an open question.

\subsection{Some illustrative examples}

Theorem A is applicable to any subshift; however, we highlight specific examples of particular interest, especially in light of the motivations outlined in the introduction.

We note that our results have interesting implications already in one dimension ($d=1$). For example, the objective of \cite{BL1} and \cite{BL2} was
to determine for which sequences $(a_n)_{n\geq 1}$ the Hofbauer-type potential $\phi(x)=a_n$ if $\dist(x,X_0)=2^{-n}$ has a freezing phase 
transition given that $X_0$ is the Thue-Morse shift or the Fibonacci subshift respectively. Their conclusion is that $\phi$ freezes when $a_n\approx -\frac{1}{n^\alpha}, \alpha \in (0,\,1)$ for the former and when $a_n\approx -\frac{1}{n}$ for the latter. Their techniques rely 
heavily on the specific properties of the Thue-Morse subshift and Fibonacci shift, which are very well understood, and the use of the transfer operator on 
a certain induced system. The proofs of  \cite{BL1,BL2} appear to be incomplete, and the ongoing effort \cite{BCHL} aims to justify them in the Thue-Morse case as well as establish freezing when $a_n\approx -\frac{1}{n}$.  Meanwhile, a direct application of Theorem \ref{thm:existence} (for the two-sided Thue-Morse and Fibonacci subshifts) 
gives that $\phi$ freezes when $a_n\approx -\frac{\log n}{n}$. Also, Theorem B shows that $\phi$ does not freeze for $a_n\leq -\frac{1}{n^{1+\epsilon}}$ for all $n$ and for some $\epsilon>0$. We 
discuss further freezing phase transitions for  one-dimensional subshifts in Section \ref{appendix:1d}. In particular, up to a logarithmic factor, we recover the result of \cite{BCHL} on the Thue-Morse subshift using a completely different approach. Moreover, Theorem \ref{thm:super-pins} applies to any subshift.

In \cite{vEKM}, the authors consider Sturmian subshifts (which can be used for instance to code irrational
rotations on the unit circle). Here $\cA=\{0,1\}$ and these subshifts are minimal, uniquely ergodic, and have zero topological entropy. They construct an infinite-range potential, decaying as fast as one wants, whose maximizing subshift is a Sturmian subshift. Using Theorem A (see Theorem \ref{thm:existence} for a more precise statement), we know how to construct infinite-range 
potentials, decaying slowly enough, such that a freezing phase transition occurs. 
However, using Theorem B, we know that the potential in \cite{vEKM} cannot have a freezing phase transition.

We now turn to higher-dimensional subshifts, focusing on subshifts of finite type (SFTs) which exhibit significantly greater diversity compared to their one-dimensional counterparts \cite{LS},
and are relevant to model quasicrystals, as mentioned in the introduction. Informally, $X_0\subset \cA^{\Z^d}$ is an SFT if no configuration in $X_0$ contains a pattern from a prescribed collection
of patterns $F$, where $F\subset \cA^{\Lambda}$ and $\Lambda\subset\Z^d$ is finite. 

The first example we mention is taken from \cite[Example 1.5]{BS-2008}. Let $M \in \N$ and define $\cA = \{-M, -M+1, \ldots, -2, -1, 1, 2, \ldots, M-1, M\}$. Then, $x \in X_0$ if and only if, whenever $\|\bi - \bj\|_1 = 1$, we have $(x_\bi, x_\bj) \notin \{(a, a') \in \cA \times \cA : aa' \leq -2\}$. In other words, the only way for a negative number to have a positive neighbor is through the pair $\pm 1$. This defines a (strongly) irreducible subshift of finite type (SFT), which consequently has positive topological entropy. Theorem 1.17 in \cite{BS-2008} then asserts that if $M > 4\e 28^d$, there are exactly two extremal (ergodic) measures of maximal entropy.

In dimension $d \geq 2$, there exist subshifts of finite type (SFTs) that have uncountably many configurations but no periodic configurations -- a phenomenon that is impossible for $d = 1$. Let us focus on the case $d = 2$. Such SFTs are typically constructed using aperiodic Wang tilings. Given a finite collection of $1\times 1$ square tiles with colored edges, copies of the tiles are placed (without rotation or reflection) at the points of $\mathbb{Z}^2$, arranged so that adjacent edges match in color. For further details, see \cite{LS}.  
A notable example is the Kari-Culik shift, which is constructed in a fundamentally different way from classical examples and, moreover, has positive topological entropy \cite{DGG2017}. According to a folk theorem, this SFT cannot be uniquely ergodic. Recently, a family of minimal, uniquely ergodic, aperiodic Wang shifts was constructed in \cite{Labbe-2023, Labbe-2024}. These shifts are self-similar in the sense that they can also be obtained using two-dimensional substitutions, where the inflation factors are the $n$-th metallic means. This provides an elegant family of quasicrystal models for which we have explicit constructions of potentials that freeze on them.

\section{Freezing on any proper subshift}
\label{sec:freezing}

\subsection{Notation and definitions}
Let $X=\cA^{\Z^d}$ be the $d$-dimensional full shift on a finite 
alphabet $\cA$, {\em i.e.}, $X$ is the space of all maps
$x:\Z^d\to \cA$.
We will call them configurations.
We put the discrete topology on $\cA$ and the product topology on $X$ (which is thus  compact).
For $\Lambda\subset \Z^d$, $x\vert_{\Lambda}$ is the projection of $x$ onto $\Lambda$ (or the restriction of $x$
to $\Lambda$ as a map).
For $\bn\in \Z^d$ we define the shift map $\sigma^{\bn}$ of $X$ 
by $(\sigma^{\bn} x)_{\bm}=x_{\bn+\bm}$ for $x=(x_{\bm})_{\bm\in \Z^d}\in X$.
The shift map is the natural action of $\Zd$ on $X$ by translations.
We take $\dist(x,y)=2^{-n}$, where
$n=n(x,y)=\min\{\Vert\bn\Vert_\infty:x_{\bn}\neq y_{\bn}\}$, as the
distance between two configurations $x,y$ (compatible with 
the product topology).

The space of shift-invariant probability measures is equipped with the vague topology, under which it is compact.

A subshift $X_0$ of $X$ is a closed subset $X_0\in X$ which is 
$\sigma^{\bn}$-invariant for every $\bn\in\Z^d$. Given $\Lambda\Subset\Zd$ 
(that is, $\Lambda$ is a finite subset of $\Z^d$), 
a pattern $w\in \cA^\Lambda$ appears in a configuration $x$ if 
there exists $\bm_{0}\in\Zd$ such that $w_{\bm}=x_{\bm+\bm_0}$ for all 
$\bm\in\Lambda$. We recall the definition of the language of $X_0$.
For $\Lambda\Subset \Zd$ , define $\cL_\Lambda(X_0)$ as the 
set of patterns supported on $\Lambda$ that appear in some configuration
$x\in X_0$. The language of $X_0$ is defined as 
$\cL(X_0)=\bigcup_{\Lambda\Subset \Zd}\cL_\Lambda(X_0)$.


Throughout the paper, we define 
$\ergbox n=[0,n-1]^d\cap \Zd$ and $\statbox n=[-n,n]^d\cap \Zd$; these
appear more frequently in ergodic theory and statistical physics respectively. 
For $\Lambda\subset \Z^d$, we denote by $|\Lambda|$ its volume, that is, the number of sites
in $\Lambda$.
The topological entropy of $X_0$ is 
\begin{equation}\label{eq:top_entropy}
\ent_{\topo}(X_0)=\lim_{n\to\infty}
\frac{1}{\vert\ergbox n\vert}\log \vert \cL_{\ergbox n}(X_0)\vert.
\end{equation}
A function $\phi:X\to\R$ is continuous if and only if  
$\var_n(\phi):=\sup\{|\phi(x)-\phi(y)|: x\vert_{\statbox n}=y\vert_{\statbox n}\}$ goes to $0$. 
(Note that $\var_n(\phi)=\sup\{\vert\phi(x)-\phi(y)\vert : x,y\in X, \dist(x,y)\le 2^{-n}\}$, the modulus of continuity of $\phi$.)
Given $\phi:X\to\R$ continuous, recall that $\mu\in \ES(\phi)$ means that
\[
\ent(\nu)+\int \phi \dd\mu=\sup\left\{ \ent(\nu)+\int \phi \dd\nu: \nu\;\text{shift-invariant}\right\},
\]
where $\ent(\nu)$ is the (measure-theoretic) entropy of $\nu$, that is,
\[
\ent(\nu)=\lim_{n\to\infty}-\frac{1}{|\ergbox n|}\sum_{w\,\in\cA^{\ergbox n}}
\nu([w])\log \nu([w])\,
\]
with $[w]=\left\{x\in X:x_{\ergbox n}=w\right\}$. 
By compactness in vague topology, there exists at 
least one equilibrium state for $\phi$, that is, at least one $\mu$ maximizing the 
$\ent(\nu)+\int \phi \dd\nu$, as a function of $\nu$; this is the variational principle.
By definition, the above supremum is the topological pressure of $\phi$, denoted 
$P(\phi)$. If $\phi\equiv 0$ then $P(\phi)=P(0)=\ent_{\topo}(X)=\log|\mathcal{A}|$. 
(More generally, this is true if $\phi=u\circ \sigma-u+c$ where $u:X\to\R$ is continuous, $c\in\R$, see \cite[p. 61]{Ruelle}.)
Given $\beta\in\R$, we can define, as done in the introduction, the equilibrium states of $\beta\phi$ in the obvious way,
as well as the pressure function $\beta\mapsto p_\phi(\beta):=P(\beta\phi)$ (which is convex and Lipschitz continuous).
Given a subshift $X_0$ of $X$, its measures of maximal entropy are the measures whose measure-theoretic entropy is the same as $\ent_{\topo}(X_0)$.

\subsection{Dyadic tiling lemma}
The following lemma plays a key role in our construction of potentials that freeze on a subshift. 
To show that equilibrium states are supported on the given subshift, we will show that ergodic measures
not supported on the subshift do not achieve the supremum in the variational principle. Our 
approach to this is based on efficiently partitioning the orbit of a point lying outside
the subshift into blocks that do lie inside the subshift. 

\begin{defn}
A \emph{tiling} $t$ of $\Z^d$ is a partition of $\Z^d$ into a disjoint union of finite non-empty subsets: $\Z^d=\bigcup_{i\in I}S_i$. 
Given $\bj\in\Z^d$, we let $t(\bj)$ denote the set $S_i$ containing $\bj$.
\end{defn}

We let $\mathcal T(\Z^d)$ denote the collection of all tilings of $\Z^d$. If $t=\{S_i\colon i\in I\}$ is a tiling
and $\bj\in\Z^d$, we define the \emph{translation of $t$ by $\bj$} to be $\mathsf{Tr}_{\bj}(t)=\{S_i-\bj\colon i\in I\}$, where 
$S_i-\bj$ denotes the translate of $S_i$ by the vector $\bj$, that is, 
$S_i-\bj=\{\bk-\bj\colon \bk\in S_i\}$.

\begin{defn}
    Let $T$ be a measurable $\Z^d$-action on a measure space $(Z,\rho)$. An \emph{equivariant tiling} 
    is a map $\tau$ from $Z$ to $\mathcal T(\Z^d)$ with the property that $\tau(T^{\bj}z)=\mathsf{Tr}_{\bj}(\tau(z))$.
\end{defn}
Rather than writing $\tau(z)$ for the tiling, we will write $\tau^z$, so that $\tau^z(\bj)$ is the tile
in $\tau(z)$ containing $\bj$.

We describe the standard dyadic odometer as the space $\Sigma=\{0,1\}^{\N_0}$, where
we write a point $y\in \Sigma$ as a left-infinite sequence $\ldots y_2y_1y_0$ and define the homeomorphism $\iota$ of $\Sigma$ by
\[
\iota(y)=\begin{cases}
\ldots y_{n+1}10\ldots0&\text{if $y_n\ldots y_0=01\ldots 1$ for some $n\ge 0$};\\
\ldots 0000&\text{if $y=\ldots 1111$}.
\end{cases}
\]
That is, in our presentation of the standard odometer, the map $\iota$ is binary addition of 1 with carry to the left. 
Given $y\in\Sigma$, we define a family of equivalence relations on $\Z$ by $i\uset{\widesim}{n,y} j$ if 
$\iota^i(y)$ and $\iota^j(y)$ agree in all coordinates at level $n$ and above. 
That is, the $\uset{\widesim}{0,y}$-equivalence classes corresponding to $y$ partition $\Z$ into singletons; 
the $\uset{\widesim}{1,y}$-equivalence classes corresponding to $y$ partition $\Z$ into intervals of length
2;
and the $\uset{\widesim}{n,y}$-equivalence classes corresponding to $y$ partition $\Z$ into intervals of length $2^n$. 
Further, each $\uset{\widesim}{n,y}$-equivalence class is the union of exactly two
$\uset{\widesim}{n-1,y}$-equivalence classes. 
If $\tau_n$ is the map sending $y$ to the collection of
$\uset{\widesim}{n,y}$-equivalence classes, we see that each
$\tau_n$ is a $\Z$-equivariant tiling.

\begin{defn}
The \emph{$\Z^d$-dyadic odometer} is the space
\[
Y=\{(y_1,\ldots,y_d)\colon y_i\in \Sigma\},
\]
with the $\Z^d$-action on $Y$ given by $T_Y^{\bj}(y)=(\iota^{j_1}(y_1),\ldots,\iota^{j_d}(y_d))$.
\end{defn}

The $\Z^d$-dyadic odometer has a family of natural equivariant tilings given by
\[
\tau_n^y=\{S_n(\bj,y)\colon \bj\in \Z^d\},
\]
where 
\[
S_n(\bj,y)=\big\{\bk\in \Z^d\colon j_i\uset{\sim}{n,y_i}k_i\text{ for each $1\le i\le d$}\big\}.
\]
That is, $\tau^y_n$ is a decomposition of $\Z^d$ into a grid of $2^n\times\cdots\times 2^n$ boxes, and each $\tau_n^y$ is the 
union of $2^d$ $\tau^y_{n-1}$ boxes. 

The following {\em Dyadic Tiling Lemma} is closely related to a result that appears in the work of Aaronson and Weiss \cite{AW}; see also notes of Meyerovitch \cite{Meyerovitch} for related work. 

\begin{lem}[Dyadic Tiling Lemma]\label{lem:tiling}
Let $X_0$ be a proper subshift of $X=\cA^{\Z^d}$ and let $\mu$ be an ergodic shift-invariant measure 
supported on $X\setminus X_0$. Let $Y$ be the dyadic $\Z^d$-odometer with unique invariant
measure $\nu$ and let $\bar\mu$ be an ergodic joining of $\mu$ and $\nu$, where 
$\Z^d$ acts by the product action, $T^{\bj}(x,y)=\big(\sigma^{\bj}(x),T_Y^{\bj}(y)\big)$.

Then there exists an equivariant tiling $\tau\colon X\times Y\to \cT(\Z^d)$ with the following properties:
\begin{enumerate}
    \item For $\bar\mu$-a.e. $(x,y)$ and each $\bj\in\Z^d$, there exists
    $n\in\N_0$ such that $\tau^{(x,y)}(\bj)=\tau_n^y(\bj)$. That is, all tiles appearing in 
    $\tau(x,y)$ are dyadic tiles defined by $y$ (at various scales).
    \item For $\bar\mu$-a.e. $(x,y)$ and each $\bj\in\Z^d$,
    if $\tau^{(x,y)}(\bj)=\tau_n^y(\bj)$ and $n>0$, then $x\vert_{\tau^{(x,y)}(\bj)}\in\cL(X_0)$. That is, the
    restriction of $x$ to any $\tau^{(x,y)}$-tile that is larger than $1\times\cdots\times 1$
    belongs to the language of $X_0$;
    \item For $\bar\mu$-a.e. $(x,y)$ and each $\bj\in\Z^d$,
    if $\tau^{(x,y)}(\bj)=\tau_n^y(\bj)$, then $x|_{\tau_{n+1}^y(\bj)}\not\in\cL(X_0)$. That is, the
    restriction of $x$ to any dyadic tile larger than the $\tau^{(x,y)}$-tile does not belong to the language of $X_0$.
    \label{it:maximal}
\end{enumerate}
\end{lem}
Taken as a whole, the lemma states that the tiling $\tau(x,y)$ is dyadic and each tile larger than the 0th
level is the maximal dyadic tile for which the corresponding $x$ word belongs to the language of $X_0$. 

\begin{proof}
Let $\mu$ and $\nu$ be as in the statement and let $\bar\mu$ be an ergodic joining of them. 
Define 
\begin{equation}\label{eq:def_n(x,y)}
    n^{(x,y)}(\bj)=\begin{cases}
0&\text{if $x_{\bj}\not\in \cL(X_0)$;}\\
\max\{n\in\N_0\colon x|_{\tau_n^y(\bj)}\in\cL(X_0)\}&\text{otherwise}.
\end{cases}
\end{equation}
Note that since $\mu$ is ergodic and $\mu$ is not supported on $X_0$, $n^{(x,y)}$ is
almost surely finite.
We then define the tiling $\tau$ by
\[
\tau^{(x,y)}(\bj)=\tau_{n^{(x,y)}(\bj)}^y(\bj)\text{ for all $(x,y)$ and $\bj$}.
\]
That is, $\tau^{(x,y)}(\bj)$ is the largest dyadic tile (as determined by $y$) for
which the corresponding $x$-block belongs to $\cL(X_0)$.
To see that this is a tiling, let $(x,y)\in X\times Y$, let $\bj\in \Z^d$ and let $n=n^{(x,y)}(\bj)$. 
First notice that $\bj\in\tau^{(x,y)}(\bj)$ for each $\bj\in\Z^d$,
so that the tiles cover $\Z^d$. 
Next note that if $\bk\in\tau^{(x,y)}(\bj)$, then
$\tau_{n}^y(\bj)=\tau_{n}^y(\bk)$, so that $\tau_m^y(\bj)=\tau_m^y(\bk)$ for all $m\ge n$.
From the definition of $n^{(x,y)}$, it follows that $n^{(x,y)}(\bk)=n^{(x,y)}(\bj)$,
so that $\tau^{(x,y)}(\bk)=\tau^{(x,y)}(\bj)$. This ensures that distinct tiles do not
overlap. The three properties then follow by
construction of $\tau^{(x,y)}$.
\end{proof}

\subsection{Potentials that freeze on a given subshift}
We will prove Theorem A, which, as stated above, follows as a consequence of a more detailed theorem. 
The following proposition provides sufficient conditions for a potential to freeze on a given subshift.
\begin{prop}\label{lem:special-class}
Let $X_0$ be a proper subshift of $X$ and $\phi\colon X\to\R$ a continuous potential
such that
\begin{enumerate}
\item $\phi\vert_{X_0}\equiv 0$ and $\phi(x)<0$ for all $x\notin X_0$;\label{prop_one}
\item $p_\phi(1)=\ent_{\topo}(X_0)$;\label{prop_two}
\item The set of equilibrium states for $\phi$ is exactly the set of measures of maximal entropy for $X_0$.
\label{prop_three}
\end{enumerate}
Then there exists $\beta_c\in (0,1]$ such that 
$\phi$ has a freezing phase transition at $\beta_c$.
For $\beta>\beta_c$, $\ES(\beta\phi)$ is the set of measures of maximal entropy on $X_0$;
while for $\beta<\beta_c$, $\mu(X_0)=0$ for all $\mu\in\ES(\beta\phi)$. 
In particular, if $X_0$ is the smallest subshift containing the support of all 
measures of maximal entropy on $X_0$, then $\phi$ freezes on $X_0$ at $\beta_c$.

\end{prop}
\begin{proof}
Since $\phi\le 0$, $p_\phi(\beta)$ is non-increasing, so for $\beta\geq 1$ 
$p_\phi(\beta)\le p_\phi(1)=\ent_{\topo}(X_0)$. If $\mu$ is any measure
of maximal entropy on $X_0$, then since $\phi\vert_{X_0}\equiv 0$,
we have $\ent(\mu)+\beta\int\phi\dd\mu=\ent_{\topo}(X_0)$, so that 
$p_\phi(\beta)=\ent_{\topo}(X_0)$ for any $\beta\ge 1$.

On the other hand, for $\beta=0$ we have 
$p_\phi(\beta\phi)=p_\phi(0)=\ent_{\topo}(X)>\ent_{\topo}(X_0)$, since $X_0$ is a proper subshift 
of $X$. Let $\beta_c=\inf\{\beta\colon p_\phi(\beta)=\ent_{\topo}(X_0)\}$, so that $0<\beta_c\le 1$.
By Proposition \ref{prop:freezing=slant-asymptote}, $\phi$ has a freezing phase transition at $\beta_c$.
By Remark \ref{rem:ESdisjoint}, $\ES(\beta\phi)$ is the set of measures of maximal entropy
on $X_0$ for $\beta>\beta_c$.

Hence, in the case that $X_0$ is the smallest subshift containing the supports of all measures
of maximal entropy on $X_0$, we have shown $\phi$ freezes on $X_0$ at $\beta_c$. 

Finally, if $0<\beta<\beta_c$, then $p_\phi(\beta)>\ent_{\topo}(X_0)$
so $\ES(\beta\phi)$ does not contain any measures supported on $X_0$. 
By ergodicity, for $\mu\in\ES(\beta\phi)$, $\mu(X_0)=0$. 

\end{proof}

\begin{figure}[htb!]
\centering
\includegraphics[scale=.25]{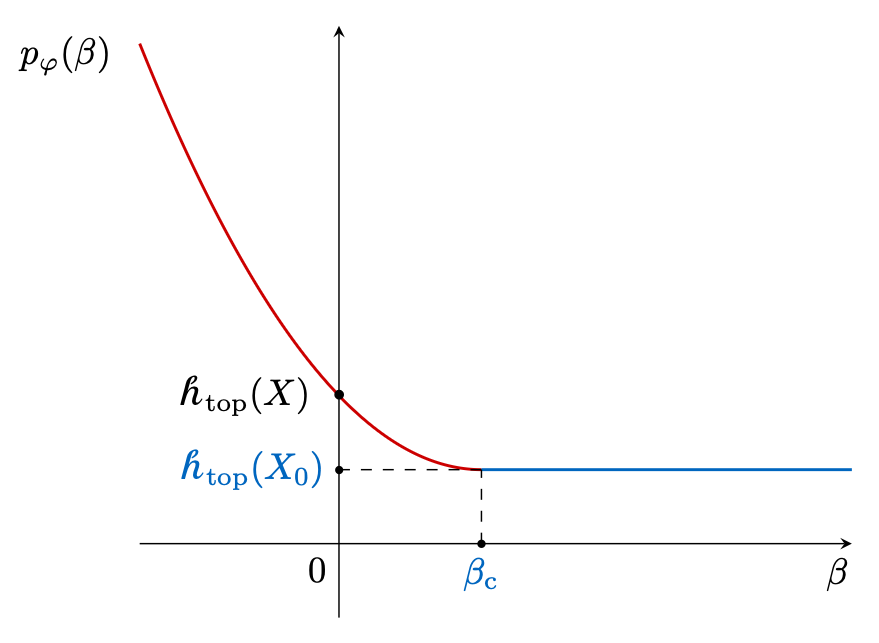}
\caption{Freezing on a proper subshift $X_0$.}
\label{fig:fpt-on-X0}
\end{figure}

As the next theorem shows, it is indeed possible to construct potentials satisfying the assumptions of
the previous proposition, with a rather precise control on their variation.

\begin{thm}\label{thm:existence}
  Suppose $X=\cA^{\Z^d}$ is a d-dimensional full shift on a finite alphabet $\cA$ and $X_0$ is a proper subshift of $X$. Define
\begin{equation}\label{eq:def_phi}
  \phi(x)=
\begin{cases}        
0 & \text{if} \quad x\in X_0,\\
-a_j & \text{if} \quad \dist(x,X_0)= 2^{-j},
\end{cases}
\end{equation}
where $(a_j)$ is a decreasing sequence of positive numbers, converging to zero and satisfying
the condition
\[
a_{2^{i+1}}\ge \left(\frac{\log n_i}{2^{id}}-\ent_\mathrm{top}(X_0)\right)+\frac{2\log^+i+3}{2^{id}}\text{ for all $i\in\N_0$},
\]
where $n_i$ denotes the number of admissible $2^i\times\dots\times 2^i$ blocks in $X_0$ and $\log^+(\cdot)=\max(\log x,0)$.
Then $\phi$ satisfies the conditions of Proposition \ref{lem:special-class}.
\end{thm}
\begin{rem}\label{rem:freezing-Phi}
An interaction $\Phi$ can be associated with the $\phi$ given in
\eqref{eq:def_phi}, namely
\[
\Phi_{\statbox n}(x)=\begin{cases}
    a_{n+1}-a_n&\text{if $x|_{\statbox n}\not\in\cL_{\statbox n}(X_0)$;}\\
    0&\text{otherwise.}
\end{cases}
\]
This interaction lies in the space $\mathcal{B}$ (see~\eqref{BigBanachSpace} for the definition), but clearly not in the space $\mathcal{S}$ of absolutely summable interactions (see~\eqref{SmallBanachSpace}). Indeed, as we demonstrate, absolute summability precludes the possibility of a freezing phase transition (see Theorem~B' above).
\end{rem}


 
Notice that $\frac{\log n_i}{2^{id}}-\ent_\text{top}(X_0)$ is non-increasing and converges to 0 by sub-multiplicativity of the pattern-counting function 
and the definition of topological entropy; and also $(2\log^+i+3)/2^{id}$ decreases to 0.
If we define $(a_j)$ by
\[
a_j=\left(\frac{\log n_i}{2^{id}}-\ent_\text{top}(X_0)\right)+\frac{2\log^+i+3}{2^{id}}\text{\quad when $2^i<j\le 2^{i+1}$}
\]
as above, then we note that $\var_j\phi= a_j$. In particular, we obtain 
$$
\var_j\phi=\kappa_{\lfloor\log_2j\rfloor}+O\left(\frac{\log\log j}{j^d}\right),
$$
where $\kappa_i=\frac1{2^{id}}\log n_i-\ent_\text{top}(X_0)$.
That is, $\kappa_i$ measures the rate of convergence of $|\ergbox{2^i}|^{-1}\log |\cL_{\ergbox{2^i}}(X_0)|$ to $\ent_\text{top}(X_0)$. 

For subshifts generated by substitutions, we have the following estimate.

\begin{corollary}\label{cor:entropy_gap}
    For a constant-box $\Z^d$-substitution shift $X_0$ there is a potential $\phi$ of
    the form of (\ref{eq:def_phi}) with $\var_j\phi=O\big(\frac{\log j}{j^d}\big)$ such that $\phi$ has a 
    freezing phase transition on $X_0$.
\end{corollary}

Observe that for any $d$, if $\var_j\phi$ is as in the corollary, 
then $\sum_j j^{d-1} \var_j\phi=\infty$  (compare with Theorem B).

\begin{proof}
It is known \cite[Theorem 7.17]{Robinson} that a $\Z^d$-substitution
dynamical system obtained from a substitution rule where each symbol is mapped to a $m_1\times m_2\times\cdots\times m_d$ box
has a word complexity function given by $|\cL_{\ergbox n}(X_0)|=O(n^{\log(m_1\cdots m_d)/\log\min m_j})$. 
Since \cite{Robinson} does not include a proof but refers to an unpublished thesis, we include a proof of this below. 
Given the estimate, we deduce
\[
\kappa_i=\frac 1{2^{id}}\log O(2^{i(\log(m_1\cdots m_d)/\log\min m_j)})=O(i/2^{id}).
\]
It follows that $a_n=O\big(\log n/n^d\big)$ as claimed.

The bound for the number of words may be obtained as follows. 
Let $m=\min(m_1,\ldots,m_d)$ and write $m_j=m^{a_j}$ for each $j$ so that $a_j=\log m_j/\log m$ and let $Q=a_1+\cdots+a_d$.
Now given $n$, let $k$ be the smallest natural number such that $m^k\ge n$, that is $k=\lceil \log n/\log m\rceil$.
If $x$ is any point in $X_0$ and $\bq\in\Z^d$, then $x_\bq$ lies in the $k$-fold substitution of some symbol $a$.
The word $x_{\bq+\ergbox n}$ may intersect up to $2^d$ $k$-fold substitution boxes, since each such box has dimensions
at least $n$ in each direction. If $\boldsymbol{o}$ is the origin of the box containing $x_{\bq}$, the $2^d$ boxes  that may be
intersected are based at $\boldsymbol{o}+\epsilon_1m_1^k\boldsymbol{e}_1+\cdots +\epsilon_d\,m_d^k\boldsymbol{e}_d$ where
$\epsilon_i\in\{0,1\}$ for each $i$ and $\mathbf e_i$ is the $i$th standard basis vector.
Thus $x_{\bq+\ergbox n}$ is fully 
determined by the choice of the symbols that were $k$-times substituted in the $2^d$ boxes 
and the difference $\bq-\mathbf o$. 
Hence we obtain the bound $|\cL_{\ergbox n}(X_0)|\le |A|^{2^d}(m_1\cdots m_d)^k
=|A|^{2^d}(m^Q)^k=|A|^{2^d}\big(m^k\big)^Q$. Since $m^k<mn$, we deduce $|\cL_{\ergbox n}(X_0)|\le|A|^{2^d}m^Qn^Q$,
so that $|\cL_{\ergbox n}(X_0)|=O\big(n^Q\big)$ as claimed.
\end{proof}

\begin{proof}[Proof of Theorem \ref{thm:existence}]
Let $\phi$ be a potential on $X$ as described in the statement of the theorem. Since $\phi\vert_{X_0}=0$, 
if $\mu_0$ is a measure of maximal entropy on $X_0$, then $\ent(\mu_0)+\int \phi\dd\mu_0=\ent_\text{top}(X_0)$. 
To complete the proof, we show that if $\mu$ is any ergodic measure not supported on $X_0$, then
$\ent(\mu)+\int \phi\dd\mu<\ent_\text{top}(X_0)$. 

Let $\mu$ be an ergodic measure on $X$ such that the support of $\mu$ does not lie in $X_0$ (so that by ergodicity, $\mu(X_0)=0$). 
We let $\bar\mu$ be an ergodic joining of $\mu$ and $\nu$ where $\nu$ is the unique invariant measure on the 
dyadic $\Z^d$ odometer. By Lemma \ref{lem:tiling}, there exists a $\bar\mu$-a.e.\ defined
equivariant tiling $\tau^{(x,y)}$ of $\Z^d$ satisfying the properties given in that lemma.

Let $E=\{(x,y)\colon \lexmin \tau^{(x,y)}(\mathbf 0)=\mathbf 0\}$. 
For each $i\in\N_0$, 
let $E_i=\{(x,y)\in E\colon n^{(x,y)}(\mathbf 0)=i\}$ with $ n^{(x,y)}(\mathbf 0)$ as defined in (\ref{eq:def_n(x,y)}),
let $n_i=|\cL_{\ergbox{2^i}}(X_0)|$ denote the number of 
legal $2^i\times \cdots\times 2^i$-blocks in $\cL(X_0)$, and
let $W^i_1,\ldots ,W^i_{n_i}$ be an enumeration of these blocks.
We let $E_{i,j}=\{(x,y)\in E_i\colon x\vert_{\ergbox{2^i}}=W^i_j\}$.
We now define three partitions:
\begin{align*}
\cP_0&=\{E,E^\comp\};\\
\cP_1&=\{E^\comp,E_0,E_1,\ldots\};\text{ and}\\
\cP_2&=\{E^\comp,E_{i,j}\colon i\in\N_0,\ 1\le j\le n_i)\}.
\end{align*}
Let $p_i=\bar\mu(E_i)/\bar\mu(E)$. Clearly, $\sum_{i=0}^\infty p_i=1$. 
We claim that additionally $\bar\mu(E)\sum_i 2^{di}p_i=1$. 
To see this second equality, notice that since $\tau^{(x,y)}$ is almost surely a tiling of $\Z^d$,
the sets $\{T^{\bj}(E_i)\colon \bj\in\ergbox{2^i}\}$ form a partition (up to sets of measure 0) of $X\times Y$. 
Summing the measures of the sets, we obtain
$1=\sum_{i=0}^\infty 2^{di}\bar\mu(E_i)=\bar\mu(E)\sum_{i=0}^\infty 2^{di}p_i$.
In the next two lemmas, we estimate $\ent(\mu)$ and $\int\phi\dd\mu$
in terms of the quantities we have introduced.

\begin{lem}\label{lem:hest}
Let $\mu$ be an ergodic invariant measure on $X$ that is not supported on $X_0$. With the notation defined above,
\begin{align*}
\ent(\mu)&
<\bar\mu(E)\sum_{i=0}^\infty p_i(3+2\log^+ i+\log n_i).
\end{align*}
\end{lem}

\begin{lem}\label{lem:phiest}
Let $\mu$ be an ergodic invariant measure on $X$ that is not supported on $X_0$. With the notation defined above,
\begin{align*}
    \int \phi\dd\mu\le-\bar\mu(E)\sum_{i=0}^\infty p_i\,2^{id}a_{2^{i+1}}.
\end{align*}
\end{lem}

We now complete the proof of Theorem \ref{thm:existence}, deferring the proof of the lemmas until afterwards.

Summing the two contributions from the lemmas, we see 
\begin{align*}
&\ent(\mu)+\int\phi\dd\mu< \bar{\mu}(E)\sum_{i=0}^\infty p_i\big(3+2\log^+i+\log n_i-2^{id}a_{2^{i+1}}\big)\\
&=\ent_\text{top}(X_0)+\bar{\mu}(E)\!\sum_{i=0}^\infty p_i\big(3+2\log^+i + \log n_i
-2^{id}\ent_\text{top}(X_0)-2^{id}a_{2^{i+1}}\big).
\end{align*}
By hypothesis, all the terms in the sum are non-positive, so that
$\ent(\mu)+\int\phi\dd\mu<\ent_\text{top}(X_0)$ for any ergodic $\mu$ that
is not supported on $X_0$. 
\end{proof}

\begin{proof}[Proof of Lemma \ref{lem:hest}]
Notice that $\cP_0\prec\cP_1\prec \cP_2$. Let $\cQ_X=\{[a]\times Y\colon a\in\cA\}$ be the partition according to the symbol $x_{\mathbf 0}$ and let $\cB_X=\bigvee_{\bj\in \Z^d}T^\bj\cQ_X$. Since $\bar\mu$ is a joining of $\mu$ and another measure, we have $\ent(\bar\mu,\cQ_X)=\ent(\mu)$. 
We claim that the elements of $\cQ_X$ (and hence of $T^\bj\cQ_X$ and hence of $\cB_X$)
agree with elements of $\bigvee_{\bj\in\Z^d}T^\bj\cP_2$ up to measure 0 sets.
For $\bar\mu$-a.e. $(x,y)$, $\bigvee_{\bj\in\Z^d}T^\bj\cP_0$ determines the tiling $\tau^{(x,y)}$;
and $\bigvee_{\bj\in\Z^d}T^\bj\cP_2$ determines $x_{\mathbf 0}$.
We define $\cB^-=\bigvee_{\bj<_\text{lex}\mathbf 0}T^\bj\cP_2$,
where $<_\text{lex}$ refers to the standard lexicographic (total) ordering on $\Z^d$. 

By the above, we see $\ent(\mu)\le \ent(\bar\mu,\cP_2)\le \ent(\bar\mu)\le \ent(\mu)$,
where the last inequality comes from that fact that $Y$ has zero topological entropy. 
Hence we see $\ent(\mu)=\ent(\bar\mu)=\Ent_{\bar\mu}(\cP_2|\cB^-)$.

We now estimate
\begin{equation}\label{eq:hcpts}
\begin{split}
\ent(\mu)
&=\Ent_{\bar\mu}(\cP_2|\cB^-)\\
&=\Ent_{\bar\mu}(\cP_0|\cB^-)+\Ent_{\bar\mu}(\cP_1|\cB^-\vee \cP_0)+\Ent_{\bar\mu}(\cP_2|\cB^-\vee\cP_0\vee\cP_1)\\
&\le \Ent_{\bar\mu}(\cP_0|\cB^-)+\Ent_{\bar\mu}(\cP_1|\cP_0)+\Ent_{\bar\mu}(\cP_2|\cP_1).
\end{split}
\end{equation}

We observe that $\cP_0$ is $\cB^-$ measurable, as the locations and sizes of the tiles in $\tau^{(x,y)}$ starting in the lexicographic past determine whether or not a new tile is to start at the origin. 
In particular, we deduce
\[
\Ent_{\bar\mu}(\cP_0|\cB^-)=0.
\]
A standard estimate shows
\[
\Ent_{\bar\mu}(\cP_1|\cP_0)=-\bar\mu(E)\sum_{i=0}^\infty p_i\log p_i,
\]
as the partition $\cP_1$ is obtained from $\cP_0$ by splitting a set of measure $\bar\mu(E)$
into pieces that have conditional measures $p_0,p_1,\ldots$. 
For $i=0,1$, we have $-p_i\log p_i\le \frac{1}{\e}$, the maximum value of $-x\log x$.
For $i\ge 2$, we partition the terms into those $i$'s for which $p_i<\frac 1{i^2}$
and those for which $p_i\ge \frac 1{i^2}$. In the first case, since $-x\log x$ is increasing on $\big[0,\frac{1}{\e}\big]$, $-p_i\log p_i\le -\frac 1{i^2}\log\frac 1{i^2}=\frac 2{i^2}\log i$.
In the second case, we have $-\log p_i\le 2\log i$. Hence we have
\begin{align*}
    \Ent_{\bar\mu}(\cP_1|\cP_0)&\le\bar\mu(E)\left(\frac 2\e+\sum_{i\ge 2,\, p_i<1/i^2}\frac {2\log i}{i^2}+2\sum_{i\ge 2,\, p_i\ge1/i^2}p_i\log i\right)\\
    &\le \bar\mu(E)\left(\frac 2\e+\sum_{i\ge 2}\frac{2\log i}{i^2}+2\sum_{i\ge 2}p_i\log i\right)\\
    &< \bar\mu(E)\left(3+2\sum_{i\ge 2}p_i\log i\right)\\
    &=\bar\mu(E)\sum_{i\ge 0}p_i(3+2\log^+i).
\end{align*}
Similarly, we have
\[
\Ent_{\bar\mu}(\cP_2|\cP_1)\le \sum_{i=0}^\infty \bar\mu(E_i)\log n_i
=\bar\mu(E)\sum_{i=0}^\infty p_i \log n_i,
\]
since $\cP_2$ is obtained by splitting each $E_i$, a set of measure $\bar\mu(E_i)=p_i\bar\mu(E)$, into $n_i$ 
pieces. 
    
Substituting the bounds in \eqref{eq:hcpts}, we obtain
\[
\ent(\mu)<\bar\mu(E)\sum_{i=0}^\infty p_i(3+2\log^+ i + \log n_i),
\]
which completes the proof of the lemma.
\end{proof}

\begin{proof}[Proof of Lemma \ref{lem:phiest}]
Define a function $\bar\phi$ on $X\times Y$ by
$\bar\phi(x,y)=\phi(x)$. As observed above,
$\{T^\bj E_i\colon i\in\N_0,\ \bj\in\ergbox{2^i}\}$ forms a countable
partition of $X\times Y$. If $(x,y)\in T^\bj E_i$ for some 
$\bj\in\ergbox{2^i}$ then
$\bj\in\tau^{(x,y)}(\mathbf 0)=\tau^y_i(\mathbf 0)$.
By property \eqref{it:maximal}, 
$x|_{\tau^y_{i+1}(\mathbf 0)}\not\in\cL(X_0)$. It follows that $\dist(x,X_0)\ge 2^{-2^{i+1}}$ so that
$\bar\phi(x,y)=\phi(x)\le -a_{2^{i+1}}$.
Hence, we see
\begin{align*}
    \int\phi\dd\mu&=\int\bar\phi\dd\bar\mu\\
    &\le-\sum_i\sum_{\bj\in\ergbox {2^i}}a_{2^{i+1}}\bar\mu(T^{\bj} E_i)\\
    &=-\bar\mu(E)\sum_i p_i2^{id}a_{2^{i+1}}.
\end{align*}
This completes the proof.
\end{proof}

\section{Obstruction to freezing phase transitions}

\subsection{Proof of Theorem B}

In this section we show that the equilibrium states for potentials with summable variation are 
fully supported, and then use that to prove Theorem B.

\begin{thm}\label{nogo-thm}
Let $X$ be the full shift $\cA^{\Z^d}$ and let $\phi\colon X\to\R$ be a continuous potential satisfying
\begin{equation}\label{cond-summability}
\sum_{n=1}^\infty n^{d-1}\var_n\phi<\infty.
\end{equation}
Then any equilibrium state for $\phi$ is fully supported (every cylinder set has a strictly positive measure).
\end{thm}

For a sketch of the proof, if $\mu$ is a shift-invariant probability measure on $X$ which 
is not fully supported, we use $\mu$ to construct another measure $\tilde\mu$ on $X$ by occasionally 
inserting blocks that lie outside the support of $\mu$ to ensure that
$\ent(\tilde\mu)+\int\phi\dd\tilde\mu>\ent(\mu)+\int \phi\dd\mu$.
This implies that $\mu$ is not an equilibrium state of the potential $\phi$.

\begin{proof}
Let $\mu$ be an ergodic invariant measure that is not fully supported. 
We let $W$ be a block of symbols such that $\mu([W])=0$, where $[W]$ denotes
the cylinder set $\{x\colon x|_\Lambda=W\}$ and $\Lambda$ is the 
set of coordinates of symbols appearing in $W$. By extending $W$ if necessary,
we may assume $\Lambda=\ergbox n$ for some $n$.
We show that $\mu$
is not an equilibrium state for $\phi$ by constructing a shift-invariant measure $\tilde\mu$
such that
\[
\ent(\tilde\mu)+\int\phi\dd\tilde\mu>\ent(\mu)+\int \phi\dd\mu.
\]

Informally, $\tilde\mu$ will be constructed by starting from a realization, $x$, of $\mu$
and overwriting blocks in $x$ with copies of the pattern $W$ at random and with very low frequency
to obtain a new point $z$. Since $W$'s do not occur in $x$, any $W$ that occurs in $z$ occurs as a result
of this overwriting process. This allows one to recover from $z$ the original point $x$ (up to regions that were
overwritten with $W$'s), as well as the locations where the overwriting took place.
There may be some ambiguity in identifying the overwriting locations because when a $W$ is inserted
it is possible that some of the symbols of the inserted $W$ together with some of the existing symbols in $x$
may make additional copies of $W$. 
We show that the entropy of $\tilde\mu$ is bounded below by the entropy of $\mu$ \emph{plus} the entropy of
the overwriting process minus the entropy of the parts of $x$ that are overwritten and with a correction
for the ambiguity in identifying the overwrite locations. If the overwriting process has frequency $\delta$,
the entropy of $\tilde\mu$ will be $\ent(\mu)+\delta|\log\delta|-O(\delta)$. Similarly we will show
$\int \phi\dd\tilde\mu\ge \int \phi\dd\mu-O(\delta)$. Combining these, we will see that for small $\delta>0$,
the desired inequality holds. 

More precisely, let $\delta>0$ be a parameter and let $\nu_0$ be the Bernoulli measure on $Y_0=\{0,1\}^{\Z^d}$
where 1's occur with frequency $\delta$. We define a continuous shift-commuting map $\Psi\colon Y_0\to Y_0$ by
\[
\Psi(y)_{\bv}=\begin{cases}
    1&\text{if $y_\bv=1$; and $y_\bu=0$ for all $\bu$ with $0<\|\bu-\bv\|_\infty<n$}\\
    0&\text{otherwise}.
\end{cases}
\]
That is, $\Psi$ removes 1's that are closer than $n$ apart, this distance being chosen to ensure that overwritten
$W$'s don't overlap. We then define $\nu$ to be $\Psi_*(\nu_0)$. Since $\Psi$ is shift-commuting,
$\nu$ is another invariant measure on $Y_0$, and is still mixing as it is a factor of a mixing measure. We claim that $\dbar(\nu,\nu_0)\le (2n)^d\delta^2$, which is intuitively unsurprising since this is an upper bound for
the probability that $\Psi$ changes a 1 to a 0. 

Recall that the Ornstein's $\dbar$-distance between measures $\nu$ and $\nu_0$ is given by
$\dbar(\nu,\nu_0)=\inf_{\bar{\nu}} \int \1_{\{y_{\mathbf 0}\ne y_{\mathbf 0}'\}}\dd\bar{\nu}(y,y')$, where $\bar{\nu}$ runs over all joinings of $\nu$ and $\nu_0$; see {\em e.g.} \cite{Glasner}. 
To estimate $\dbar(\nu,\nu_0)$, we define a joining on $Y_0\times Y_0$ by introducing the
map $\bar{\Psi}:Y_0\to Y_0\times Y_0$ given by $\bar \Psi(y)=(y,\Psi(y))$ and letting $\bar\nu=\bar\Psi_*(\nu_0)$. 
This measure is then a joining of $\nu_0$ in the first coordinate and $\nu$ in the second. Further, one can see that
$y_{\mathbf 0}$ and $\Psi(y)_{\mathbf 0}$ (i.e., the zeroth coordinates of the two points in $\bar\Psi(y)$) only differ if $y_{\mathbf 0}=1$
and there exists $\bu$ with $0<\|\bu\|_\infty<n$ such that $y_{\bu}=1$. The probability of this is
less than $(2n)^d\delta^2$ as claimed. This shows that $\dbar(\nu,\nu_0)\le (2n)^d\delta^2$. 
Standard estimates show $|\ent(\nu)-\ent(\nu_0)|\le -\dbar(\nu,\nu_0)\log \dbar(\nu,\nu_0)-\big(1-\dbar(\nu,\nu_0))\log(1-\dbar(\nu,\nu_0)\big)$. 
This gives $\ent(\nu)\ge \ent(\nu_0)-O(\delta^2|\log\delta|)$. We then obtain the bound
\[
\ent(\nu)\ge \delta|\log\delta|+O(\delta).
\] 
Although we don't need the more refined estimate, we note that in fact $\ent(\nu)\ge \delta|\log\delta|+\delta-O(\delta^2|\log\delta|)$.

We let $Y$ denote the range of $\Psi$: the set of $y\in \{0,1\}^{\Z^d}$ such that if $y_\bu=y_\bv=1$ and $\bu\ne\bv$,
then $\|\bu-\bv\|_\infty\ge n$. 
We now define a shift-commuting map $\Xi\colon X\times Y\to X$
(that is, $\Xi(\sigma^\bv x,\sigma^\bv y)=\sigma^\bv\Xi(x,y)$ for all $\bv\in\Z^d$)
by
\[
\Xi(x,y)_\bv=\begin{cases}
    W_\bj&\text{if there exists $\bj\in\ergbox n$ such that $y_{\bv-\bj}=1$;}\\
    x_\bv&\text{otherwise.}
\end{cases}
\]
That is, $\Xi(x,y)$ replaces all blocks in $x$ starting at positions where $y_\bu=1$ by copies of $W$.
The fact that $y\in Y$ ensures that $\Xi(x,y)$ is well-defined as there cannot be two distinct $\bj$'s in $\ergbox n$
for which $y_{\bv-\bj}=1$.
Since $\nu$ is mixing, $\mu\times\nu$ is an ergodic invariant measure with respect to the action $\sigma^\bv(x,y):=
(\sigma^\bv x,\sigma^\bv y)$. We let $\tilde\mu=\Xi_*(\mu\times\nu)$, the invariant measure obtained by 
editing realizations from $\mu$ using the ``instructions" from an independent realization of $\nu$. 

We complete the proof by giving lower bounds for $\ent(\tilde\mu)$ and $\int \phi\dd\tilde\mu$. 
Let $\bar X=X\times Y\times X$ and let $\bar\Xi\colon X\times Y\to  X\times Y\times X$ be given by
$\bar\Xi(x,y)=(x,y,\Xi(x,y))$. The push-forward, $\bar\mu=\bar\Xi_*(\mu\times\nu)$ is a joining
of $\mu$, $\nu$ and $\tilde\mu$ where the first two coordinates are independent. 
We introduce a family of partitions of $\bar X$: $\cP_X=\{[a]\times Y\times X\colon a\in \cA\}$,
$\cP_Y=\{X\times [0]\times X,X\times [1]\times X\}$ and $\cP_Z=\{X\times Y\times [a]\colon a\in \cA\}$.
We also define $\cP_{XYZ}=\cP_X\vee \cP_Y\vee \cP_Z$ and $\cP_{XY}=\cP_X\vee\cP_Y$.
The $\cP_{XYZ}$ is a generating partition, so that $\ent(\bar\mu)=\ent(\bar\mu,\cP_{XYZ})$. 
Since $\bar\Xi$ is an isomorphism (its inverse is given by forgetting the third coordinate), we see that
$\ent(\bar\mu)=\ent(\mu\times\nu)=\ent(\mu)+\ent(\nu)$.
By properties of joinings 
\begin{align*}
    \ent(\bar\mu,\cP_Z)&=\ent(\tilde\mu)\\
    \ent(\bar\mu,\cP_X)&=\ent(\mu)\\
    \ent(\bar\mu,\cP_{XY})&=\ent(\mu)+\ent(\nu),
\end{align*}
where we used independence in the third equality. 
(For background on entropy, we refer to \cite{Keller}.)
We let $\cF_X$, $\cF_Y$, $\cF_Z$, $\cF_{YZ}$ and so
on be the $\sigma$-algebras generated by $\cP_X$, $\cP_Y$, $\cP_Z$, $\cP_{YZ}$ etc. 
We then use conditional entropy to write
\[
\ent(\bar\mu)=\ent(\bar\mu,\cP_Z)+\ent(\bar\mu,\cP_Y|\cF_Z)+\ent(\bar\mu,\cP_X|\cF_{YZ}).
\]
This yields
\begin{equation}\label{eq:hlowerbd}
\begin{split}
\ent(\tilde\mu)&=\ent(\bar\mu,\cP_Z)\\
&=\ent(\mu)+\ent(\nu)-\ent(\bar\mu,\cP_Y|\cF_Z)-\ent(\bar\mu,\cP_X|\cF_{YZ})\\
&\ge \ent(\mu)+\delta|\log\delta|+O(\delta)-\ent(\bar\mu,\cP_Y|\cF_Z)-\ent(\bar\mu,\cP_X|\cF_{YZ}).
\end{split}
\end{equation}
We have $\ent(\bar\mu,\cP_Y|\cF_Z)\le \Ent_{\bar\mu}(\cP_Y|\cF_Z)$. 
Note that for $\bar\mu$-a.e.\ $(x,y,z)$, if $z_{|\ergbox n}\ne W$, then $y_{\mathbf 0}=0$,
so that $I_{\bar\mu}(\cP_Y|\cF_Z)=0$ on the $\cF_Z$-measurable set $[W]^\comp$.
Since $\#\cP_Y=2$, it follows that
$\Ent_{\bar\mu}(\cP_Y|\cF_Z)\le \tilde\mu([W])\log 2$. Since each $1$ in the $Y$-layer
can create no more than $(2n-1)^d$ $W$'s 
in the $Z$-layer (each $W$ created must overlap
with the $W$ that is rewritten), we see $\tilde\mu([W])\le (2n-1)^d\nu([1])<(2n-1)^d\delta$. 
Hence 
\[
\ent(\bar\mu,\cP_Y|\cF_Z)\le (2n-1)^d\delta\log 2=O(\delta). 
\]
Similarly,
$\ent(\bar\mu,\cP_X|\cF_{YZ})\le \Ent_{\bar\mu}(\cP_X|\cF_{YZ})$.
Let $A=\bigcup_{\bj\in\ergbox n}\sigma^{\bj}(X\times [1]\times X)$, that is
$A$ is the $\cF_Y$-measurable event that $(x,y,z)$ contains a 1 in the $(-\ergbox n)$-box, which is also the event that 
$z_{\mathbf 0}$ is part of a block that was overwritten from $x$. 
We have that for $\bar\mu$-a.e.\ $(x,y,z)$, if $(x,y,z)\in A^\comp$, then $z_{\mathbf 0}=x_{\mathbf 0}$, and
$x_{\mathbf 0}$ is $\cF_{YZ}$-measurable, so
that $I_{\bar\mu}(\cP_X|\cF_{YZ})=0$ on $A^\comp$. Meanwhile, $I_{\bar\mu}(\cP_X|\cF_{YZ})\le\log\#\cA$ on $A$
so that 
$\Ent_{\bar\mu}(\cP_X|\cF_{YZ})\le \bar\mu(A)\log\#\cA$.
Since $\bar\mu(A)=n^d\nu([1])<n^d\delta$, we obtain
\[
\ent(\bar\mu,\cP_X|\cF_Y\vee\cF_Z)\le n^d\delta\log\#\cA=O(\delta).
\]
Substituting in \eqref{eq:hlowerbd}, we obtain
\begin{equation}\label{eq:htildemulower}
\ent(\tilde\mu)\ge \ent(\mu)+\delta|\log\delta|-O(\delta).
\end{equation}

To bound $\int \phi\dd\tilde\mu$ from below, we estimate $|\int \phi\dd\tilde\mu-
\int\phi\dd\mu|$. Since $\bar\mu$ is a joining of $\mu$, $\nu$ and $\tilde\mu$, we have
\begin{align*}
\left|\,\int\phi\dd\tilde\mu-\int\phi\dd\mu\,\right|&=
\left|\,\int (\phi\circ\pi_Z-\phi\circ\pi_X)\dd\bar\mu\,\right|\\
&\le \int |\,\phi\circ\pi_Z-\phi\circ\pi_X|\dd\bar\mu.
\end{align*}
Recall that for $\bar\mu$-a.e.\ $(x,y,z)\in\bar X$, the only differences between $x$ and $z$
occur in places where $y$ has a 1, when blocks of $x$ are overwritten by $W$'s. 
Letting $r(y)$ be
$\min\{\|\bj\|_\infty\colon y_{\bj}=1\}-(n-1)$ or 0 if that quantity is negative,
we have $\dist(x,z)\le 2^{-r(y)}$ for $\bar\mu$-a.e.\ $(x,y,z)$. 
Hence 
\begin{equation}\label{eq:diffbound}
\begin{split}
\left|\,\int\phi\dd\tilde\mu-\int\phi\dd\mu\,\right|&\le
\int \var_{r(y)}\phi\dd\bar\mu\\
&=\sum_{k=0}^\infty \bar\mu\{(x,y,z)\colon r(y)=k\}\var_k\phi\\
&=\sum_{k=0}^\infty \nu\{y\colon r(y)=k\}\var_k\phi.
\end{split}
\end{equation}
We observe that 
\begin{align*}
\{y\colon r(y)=0\}&=\bigcup_{\bj\in [1-n,n-1]^d}\sigma^{-\bj}[1]\quad\text{and}\\
\{y\colon r(y)=k\}&\subseteq \bigcup_{\|\bj\|_\infty=k+(n-1)}\sigma^{-\bj}[1]\quad\text{for $k>0$.}
\end{align*}
Hence
\begin{align*}
    \nu\{y\colon r(y)=0\}&\le(2n-3)^d\nu([1])\quad\text{and}\\
    \nu\{y\colon r(y)=k\}&\le \big((2k+2n-3)^d-(2k+2n-5)^d\big)\nu([1])\\
    &<2d(2k+2n-3)^{d-1}\nu([1])\text{\quad for $k>0$.}
\end{align*}
In the last inequality, we bounded the difference between the two cubes as the sum of the faces.
Substituting these bounds in \eqref{eq:diffbound},
we obtain
\[
\left|\,\int\phi\dd\tilde\mu-\int\phi\dd\mu\,\right|
\le C\nu([1]),
\]
where 
\[
C=(2n-3)^d\var_0\phi+\sum_{k=1}^\infty 2d(2k+2n-3)^{d-1}\var_k\phi.
\]
Since the sequence $2d(2k+2n-3)^{d-1}/k^{d-1}$ converges to $2^dd$ as $k\to\infty$,
there exists a constant $L>0$ such that $2d(2k+2n-3)^{d-1}\le Lk^{d-1}$ for each $k\ge 1$. 
It then follows that $C\le 2(2n-3)^d\|\phi\|+L\sum_{k=1}^\infty k^{d-1}\var_k\phi$, and so $C<\infty$ by the hypothesis. 
Hence 
\[
\left|\,\int\phi\dd\tilde\mu-\int\phi\dd\mu\,\right|\le C\delta,
\]
so that 
\begin{equation*}
\int\phi\dd\tilde\mu\ge \int\phi\dd\mu-O(\delta).
\end{equation*}
Combining this with \eqref{eq:htildemulower}, we see
\[
\ent(\tilde\mu)+\int \phi\dd\tilde\mu\ge \ent(\mu)+\int\phi\dd\mu+\delta|\log\delta|-O(\delta).
\]
In particular, for sufficiently small $\delta>0$, one sees
$\ent(\tilde\mu)+\int\phi\dd\tilde\mu>\ent(\mu)+\int\phi\dd\mu$, so that
$\mu$ is not an equilibrium state. 
\end{proof}

\begin{prop}\label{prop:notfullsupp}
Let $X$ be the full shift $\cA^{\Z^d}$ and let $\phi$ be a continuous potential satisfying
$$
\sum_{n=1}^\infty n^{d-1}\var_n\phi<\infty.
$$
If there exist shift-invariant probability measures $\mu$, $\nu$ such that $\int\phi\dd\mu\ne\int\phi\dd\nu$,
then there exists a cylinder set $C$ such that $\eta(C)=0$ for any $\eta\in\MM(\phi)$.
\end{prop}

\begin{proof}
Let $\mu_1,\mu_2$ be shift-invariant probability
measures such that $\int\phi\dd\mu_1<\int\phi\dd\mu_2$. 
Let $A=\sum n^{d-1}\var_n\phi$. We prove two simple claims. 
Firstly, for any $\ell\in\N$, if $x,x'\in X$ are such that $x|_{\ergbox \ell}=x'|_{\ergbox \ell}$, we claim
\begin{equation}\label{eq:inbound}
\sum_{\bj\in \ergbox\ell} \big|\,\phi(\sigma^{\bj} x)-\phi(\sigma^\bj x')\big| \le 2\ell^{d-1}dA.
\end{equation}
To see this, notice there are less than $2d\ell^{d-1}$ $\bj$'s in $\ergbox\ell$ that 
are at any given distance $1\le r\le \lfloor \frac\ell2\rfloor$ from $(\ergbox\ell)^\comp$,
so that the left side is bounded above by 
$2d\ell^{d-1}(\var_1\phi+\cdots+\var_{\ell/2}\phi)<2d\,\ell^{d-1}A$.

Secondly, for any $\ell\in\N$, if $x,x'\in X$ are such that $x|_{{\ergbox\ell}^\comp}=x'|_{{\ergbox\ell}^\comp}$, 
and $S$ is any finite set that is disjoint from $\ergbox\ell$, we claim
\begin{equation}\label{eq:outbound}
\sum_{\bj\in S} \big|\,\phi(\sigma^{\bj} x)-\phi(\sigma^\bj x')\big|\le 2^d\,\ell^{d-1}dA.
\end{equation}
The argument is similar: the number of $\bj$ in $S$ that are at a distance $r$ from $\ergbox \ell$
is at most $(\ell+2r)^d-(\ell+2(r-1))^d<2d(\ell+2r)^{d-1}<2d(2\ell r)^{d-1}=2^d\ell^{d-1}dr^{d-1}$.

Let $\int \phi\dd\mu_2-\int \phi\dd\mu_1=\delta>0$ and let $\ell$ be such that 
$\delta\ell^d-(2^{d}+4)\ell^{d-1}dA\ge 1$. 
Since $\int\sum_{\bj\in\ergbox\ell}\phi\dd\mu_2-\int\sum_{\bj\in\ergbox \ell}\phi\dd\mu_1=\delta\ell^d$,
there exist $u,u'\in X$ with $\sum_{\bj\in\ergbox\ell}(\phi(\sigma^\bj u')-\phi(\sigma^\bj u))\ge \delta\ell^d$. 
Let $W$ be the pattern $u|_{\ergbox\ell}$ and $W'=u'|_{\ergbox\ell}$. 

We will show that $\mu([W])=0$ for any $\phi$-maximizing measure. The idea is simple: we show that replacing copies of $W$ 
with copies of $W'$ increases the sum over the orbit. If $\mu$ is any invariant measure with $\mu([W])>0$, then
by systematically replacing copies of $W$ with copies of $W'$, one obtains a measure $\mu'$ with $\int \phi\dd\mu'
>\int\phi\dd\mu$, so that $\mu$ could not have been a maximizing measure. 

Let $x,x'$ be such that $x$ and $x'$ agree outside $\ergbox\ell$; $x_{\ergbox\ell}=W$ and $x'_{\ergbox\ell}=W'$. 
Let $S$ be any finite set such that $S\supseteq \ergbox\ell$. Then we have
\begin{align*}
    &\sum_{\bj\in S}\big(\phi(\sigma^\bj x')-\phi(\sigma^\bj x)\big)\\
    &\; \ge \sum_{\bj\in \ergbox\ell}\big(\phi(\sigma^\bj x')-\phi(\sigma^\bj x)\big)
    -\sum_{\bj\in S\setminus\ergbox\ell} \big|\,\phi(\sigma^\bj x')-\phi(\sigma^\bj x)\big|.
\end{align*}
The subtracted term on the right is at most $2^d\ell^{d-1}dA$ by \eqref{eq:outbound}. 
For the first term, we have
\begin{align*}
&\sum_{\bj\in \ergbox\ell}\big(\phi(\sigma^\bj x')-\phi(\sigma^\bj x)\big)\\
&\;\ge\sum_{\bj\in\ergbox\ell}\big(\phi(\sigma^\bj u')-\phi(\sigma^\bj u)\big)\\
&\qquad -\sum_{\bj\in\ergbox\ell}\big(|\phi(\sigma^\bj u')-\phi(\sigma^\bj x')|+|\phi(\sigma^\bj u)-\phi(\sigma^\bj x)|\big)\\
&\;\ge \delta\ell^d - 4d\ell^{d-1}A.
\end{align*}
Hence we see
\[
\sum_{\bj\in S}\big(\phi(\sigma^\bj x')-\phi(\sigma^\bj x)\big)\ge \delta\ell^d - (2^d+4d)\ell^{d-1}A\ge 1.
\]
Suppose that $\mu$ is an ergodic invariant measure for which $\mu([W])>0$. 
Let $x$ be a generic point for $\mu$ so that $\frac 1{N^d}\sum_{\bj\in \ergbox N}
\phi(\sigma^\bj x)\to \int\phi\dd\mu$ and $\frac 1{N^d}\#\{\bj\in \ergbox{N-\ell+1}\colon \sigma^\bj x\in [W]\}\to\mu([W])$. 
For each $N$, we inductively modify $x$ inside $\ergbox N$ replacing $W$'s with $W'$'s one at a time
until there are no $W$'s remaining. Let the resulting point be $y^N$.
By the previous estimate, each replacement of a $W$ by a $W'$ causes the 
sum of $\phi$ over $\ergbox N$ to increase by at least 1, so that
\[
\frac 1{N^d}\sum_{\bj\in \ergbox N}\phi(\sigma^\bj y^N)\ge
\frac 1{N^d}\sum_{\bj\in \ergbox N}\phi(\sigma^\bj x)+\frac{\#(\text{replacements})}{N^d}.
\]
Since replacing one $W$ with a $W'$ can get rid of at most $(3\ell)^d$ $W$'s,
one has $\#(\text{replacements})\ge \frac 1{3\ell^d}\,\#\{\bj\in \ergbox{N-\ell+1}\colon \sigma^\bj x\in [W]\}$.
Hence 
\[
\limsup_{N\to\infty}
\frac 1{N^d}\sum_{\bj\in \ergbox N}\phi\big(\sigma^\bj y^N\big)\ge
\int\phi\dd\mu+\frac{\mu([W])}{(3\ell)^d}.
\]
Taking a subsequential limit of the empirical measures arising from translating 
the points $y^N$ by the elements of $\ergbox N$, we obtain a shift-invariant measure $\mu'$
such that $\int\phi\dd\mu'\ge \int\phi\dd\mu+\mu([W])/(3\ell)^d$, hence
$\mu$ cannot be a maximizing measure for $\phi$. 
\end{proof}

\begin{proof}[Proof of Theorem B]
We consider two cases. First, if there exist invariant measures $\mu$ and $\nu$ for 
which $\int\phi\dd\mu\ne\int\phi\dd\nu$, then Proposition \ref{prop:notfullsupp} implies that 
maximizing measures are not fully supported, while Theorem \ref{nogo-thm} implies that
equilibrium measures for $\beta\phi$ are fully supported for all $\beta\in\R$. This implies
that there is no freezing phase transition. 

Secondly, if $\int\phi\dd\mu=\int\phi\dd\nu$ for all invariant measures, then
the unique equilibrium state is the uniform Bernoulli measure on $X$ for all $\beta$, so again there
is no freezing phase transition.
\end{proof}
We remark that by a result of Bousch \cite{Bousch-Walters},
in the case $d=1$, if $\phi$ satisfies $\sum\var_n\phi<\infty$,
then $\int\phi\dd\mu$ is equal for all invariant measures if and only if $\phi$ is 
cohomologous to a constant function.

\begin{rem}
We established that if $\phi$ satisfies \eqref{cond-summability}, then the corresponding equilibrium states have full support, meaning that every cylinder set has strictly positive measure. An alternative proof exists, though it relies on the concept of Gibbs states. This follows from Keller's book \cite[Chapter 5]{Keller}, as we now outline.  Notice that our proof takes a completely different approach, relying on ergodic theoretic techniques and notably avoiding any use of Gibbs states.

Keller first defines Gibbs states associated with a potential $\phi$. (It is worth noting that the classical definition of a Gibbs state is based on an interaction $\Phi$, as in \cite{Ruelle} and \cite{Georgii}.) He then shows that these Gibbs states -- which are not necessarily shift-invariant -- assign a strictly positive measure to all cylinder sets. This result follows from a precise estimate involving the Birkhoff sum of $\phi$ and its pressure. Finally, he demonstrates that under condition \eqref{cond-summability}, the shift-invariant Gibbs states coincide with the equilibrium states -- the key point being to establish that the equilibrium states of $\phi$ are indeed Gibbs states. Hence, the fact that equilibrium states for
$\beta\phi$ have full support follows at once for any $\beta$.
\end{rem}

\subsection{Proof of Theorem B'}\label{sec:proof-theorem-B'}
The proof scheme is the same as that of Theorem B. 
 
One first proves that if
$\sum_{\Lambda\Subset \mathbb{Z}^d,\, \Lambda \ni 0} \|\Phi_\Lambda\|_\infty < \infty$, 
then the corresponding Gibbs states for $\beta\Phi$ have full support for any $\beta$ (see Section \ref{app:Gibbs} for a proof). For this class of interactions, shift-invariant Gibbs states for $\beta\Phi$ coincide with equilibrium states for $\beta\phi$ \cite[Theorem 4.2]{Ruelle}, and consequently, the latter also have full support.

Next, \cite[Theorem 2.4 (v)]{GT} establishes that if $\mu \in \MM(\phi)$, then $\mathrm{supp}(\mu) \subseteq X_\Phi$, where $X_\Phi \subseteq X$ is the subshift of maximizing configurations for $\Phi$ (that we will not
define here)\footnote{\, Note that \cite{GT} does not explicitly state that $X_\Phi$ is a subshift.}. However, Theorem \ref{thm:folklore} tells us that $\ZTA(\phi) \subseteq \MM(\phi)$. Therefore, if $X_\Phi \subsetneq X$, a freezing phase transition cannot occur since there exist at least one cylinder
set in $X\backslash X_\Phi$.  

This leaves only the degenerate case where $X = X_\Phi$. By \cite[Theorem 2.4 (iv)]{GT}, if $\nu$ is any shift-invariant measure such that $\mathrm{supp}(\nu) \subseteq X_\Phi$, then $\nu \in \MM(\phi)$. This implies that $\int \phi \dd\nu = s_\phi$ for any shift-invariant measure $\nu$, and by the variational principle, for any $\beta \in \mathbb{R}$,  
$
p(\beta\phi) = \sup_\nu \big\{ \ent(\nu) + \beta s_\phi \big\} = \log|\mathcal{A}| + \beta s_\phi.
$
Since the pressure function $\beta\mapsto p(\beta\phi)$ is affine on $\R$, no freezing phase transition can occur in this case either.\footnote{\, In fact, if $\int \phi \dd\nu = s_\phi$ for all
shift-invariant measures $\nu$, it follows from \cite[Prop. 2.34, p. 46]{vEFS} that
$\phi-s_\phi$ belongs to the closed linear span of the family of functions $\{f-f\circ \sigma^\bi: f:X\to\R\;\text{continuous},\, \bi\in\Z^d\}$.}

\begin{rem}
Two commonly-studied (Banach) spaces of 
interactions are $\mathcal S$ and $\cB$. The properties of interactions in
these two spaces are reminiscent of the properties of potentials with summable variation and that are continuous respectively. The first one is
\begin{equation}\label{SmallBanachSpace}
\mathcal{S}:=\Bigg\{\Phi: \|\Phi\|_{\mathcal{S}}:=\sum_{\substack{\Lambda\Subset \Z^d \\ \Lambda\ni 0}}\|\Phi_\Lambda\|_\infty<\infty\Bigg\}.    
\end{equation}
Unfortunately, assuming that $\Phi\in\mathcal{S}$ does not imply that $\phi$
satisfies \eqref{cond-summability}.
A bigger (Banach) space is
\begin{equation}\label{BigBanachSpace}
\mathcal{B}:=\Bigg\{\Phi: \|\Phi\|_{\mathcal{B}}:=\sum_{\substack{\Lambda\Subset \Z^d \\ \Lambda\ni 0}} |\Lambda|^{-1}\|\Phi_\Lambda\|_\infty<\infty\Bigg\}.
\end{equation}
If $\Phi\in\mathcal{B}$ then $\phi$ is continuous. The space $\mathcal{B}$ is ``pathological'' in many respects \cite{vEFS}.    
\end{rem}

\section{The one-dimensional case}\label{appendix:1d}

Theorem A gives an explicit construction of  a potential $\phi$. 
An interesting question concerns the optimality of the potential in (\ref{eq:def_phi}) 
in a sense that defining $\phi$ using a sequence which converges to zero slower than $a_j$ 
generates a freezing phase transition, but using a faster converging sequence does not. 
For a subshift $X_0\subsetneq \cA^\Z$ we know that 
\[
{\phi(x)\approx \left(\ent_{\topo}(X_0)
-\frac{\log(\#j\text{-words in }X_0)}{j}\right)-
\frac{\log \log j}{j}}\text{ if }\,{\dist(x,X_0)=\frac{1}{2^j}}
\]
has a freezing phase transition and the zero-temperature accumulation measures for $\phi$ are the measures of maximal entropy of $X_0$.
It follows from Theorem \ref{thm:one_sided_SFT} below that when $X_0$ is an SFT, the term ${\frac{\log \log j}{j}}$ can be replaced by ${\frac{1}{j}}$. Moreover, we establish this fact in the one-sided settings. 


\begin{thm}[One-sided SFT's]\label{thm:one_sided_SFT}
  Suppose $X^+=\cA^{\N}$ is a full shift and $X_0^+\subset X^+$ is a one-step subshift of finite type. 
  Then the potential $\phi^+:X^+\to\R$ defined by   
\begin{equation}\label{eq:def_one_sided_phi}
  \phi^+(x)=
\begin{cases}        
0 & \text{if} \quad x\in X_0^+,\\
-a_j & \text{if} \quad \dist(x,X_0^+)= 2^{-j},
\end{cases}
\end{equation}
  has a freezing phase transition provided $(a_j)$ is a decreasing sequence converging to 0, such that
  $a_j\ge \frac{1}{j}$ for all $j\in\N_0$. In this case, $\ZTA(\phi)$ coincides with the set of maximal entropy measures of $X_0^+$. In particular, if all transitive components of $X_0^+$ have the same entropy then $\phi$ freezes on $X_0^+$.
\end{thm}
\begin{proof} 

 Let $\mu$ be an ergodic measure which is not supported on $X_0^+$. Then $\mu$ assigns positive measure to the set $E=\{x\colon x_0x_1\not\in\cL(X_0^+)\}$. We induce on $E$. Denote by $r(x)$ the first return time to $E$, i.e., 
$r(x)=\min\{j\ge 1: \sigma^j(x)\in E\}$. We obtain the ergodic system 
$(E,{\mu}_E,{\sigma}_E)$ with corresponding potential ${\phi}_E$ where
\[
{\mu}_E=\frac{{\mu}|_E}{{\mu}(E)}, \quad \sigma_E(x,z)
=\sigma^{r(x)}(x),\quad
{\phi}_E(x)=\sum_{i=0}^{r(x)-1}{\phi^+}({\sigma}^i(x)).
\]
Then Abramov's formula and Kac's theorem imply that 
\[
\ent({\mu})+\int{\phi^+}\dd{\mu}={\mu}(E)\left(\ent({\mu}_E)+\int{\phi}_E\dd{\mu}_E\right).
\]

Let $\cQ=\{Q_j\}_{j\in\N}$ be the first return time partition of $E$, i.e., $Q_j=\{x\in E: r(x)=j\}$, and 
let $\cP$ be its further refinement into cylinder sets. Denote by $q_j={\mu}_E(Q_j)$ and by $n_j$ the 
number of words of length $j$ which are admissible in $X_0$. We have

\[
\sum_{j=1}^{\infty}q_j=1,\quad \sum_{j=1}^{\infty}jq_j=\frac{1}{{\mu}(E)},\quad
\card\{A\in \cP\colon A\subset Q_j\}\le n_j.
\]
We estimate the entropy of measure ${\mu}_E$ by the ${\mu}_E$-entropy of the partition $\cP$:
\[
\ent({\mu}_E)\le \Ent(\cP)= \Ent(\cQ)+\Ent(\cP|\cQ)\le\sum_{j=1}^{\infty}-q_j\log q_j+\sum_{j=1}^{\infty}q_j\log n_j,
\]
where in the last step we used the fact that each $\cQ_j$ contains at most $n_j$ elements of $\cP$. 
Note that when $q_j>\frac{1}{j^2}$ we have $-\log q_j<2\log j$. Therefore,
\begin{equation}\label{eq:entropy_first_return_partition}
    \begin{split}
\sum_{j=1}^{\infty}-q_j\log q_j
& =\sum_{q_j>\frac{1}{j^\stwo}}-q_j\log q_j+\sum_{q_j\le\frac{1}{j^\stwo}}-q_j\log q_j \\
& \le \sum_{q_j>\frac{1}{j^\stwo}}2q_j\log j+\sum_{q_j\le\frac{1}{j^\stwo}}-q_j\log q_j \\
& = \sum_{j=1}^\infty2q_j\log j+\sum_{q_j\le\frac{1}{j^\stwo}}(-q_j2\log j-q_j\log q_j) \\
&= \sum_{j=1}^\infty2q_j\log j+\sum_{q_j\le\frac{1}{j^\stwo}}\frac{1}{j^2}\left(j^2q_j\log \frac{1}{q_jj^2}\right)\,.
\end{split}
\end{equation}

Since the function $x\log \frac{1}{x}$ is bounded above by $\frac{1}{\e}$ and 
$\sum_{j=1}^{\infty}\frac{1}{j^2}=\frac{\pi^\stwo}{6}$, we get
\begin{equation}\label{eq:entropy_estimate}
  \ent({\mu}_E)< \sum_{j=1}^{\infty}q_j(2\log j+\log n_j)+1.
\end{equation}
We now turn our attention to $\int{\phi}_E\dd{\mu}_E$. Recall that for $x\in Q_j$ 
we have ${\phi}_E(x)=\sum_{i=0}^{j-1}\phi^+(\sigma^ix)$. We know that $x_1\dots x_{j}$ is 
admissible in $X_0^+$, but $x_1\cdots x_{j+1}$ is not admissible in $X_0^+$ since $x_{j}x_{j-}\notin\cL(X_0^+)$.  
Therefore, $\dist(x,X_0^+)= 2^{-1}$ and $\dist(\sigma^ip(x),X_0^+)\ge 2^{-j+i}$ for $0<i<j$. Hence $\sum_{i=0}^{j-1}\phi(\sigma^i x)\le-a_1-a_{j-1}-\dots -a_2= -\sum_{i=1}^j a_i$. 
This gives
\begin{equation}\label{eq:integral_estimate}
\int{\phi}_E\dd{\mu}_E =\sum_{j=1}^{\infty}\int_{Q_j}{\phi}_E\dd{\mu}_E\le  
-\sum_{j=1}^{\infty}\left(q_j\sum_{i=1}^j a_i\right).
\end{equation}
Combining the last inequality with the entropy estimate (\ref{eq:entropy_estimate}), and using that $\sum_{j=1}^\infty q_j=1$,  we obtain
\begin{align*}
   \ent({\mu}_E)+\int{\phi}_E\dd{\mu}_E 
   & < \sum_{j=1}^{\infty}q_j \Big( 2\log j+\log n_j+1-\sum_{i=1}^j a_i\Big)\\
   & = \sum_{j=1}^{\infty}jq_j\left(\frac{2\log j}{j}+\frac{\log n_j}{j}+
   \frac{1}{j}-\frac{1}{j}\sum_{i=1}^j a_i\right)\,.
\end{align*}
Since $n_j$ is the number of words of length $j$ in $X_0^+$, it follows from Perron-Frobenius theorem that $\ent_{\topo}(X_0^+)\le \frac{\log n_j}{j}\le \ent_{\topo}(X_0^+) +\frac{c}{j}$ for some constant $c>0$.
Hence, taking $a_i=\frac{2(c+2)}{i}$ and using that $\sum_{i=1}^j \frac{1}{i}>\frac12+\log j$, we obtain
\begin{align*}
& \ent({\mu}_E)+\int{\phi}_E\dd{\mu}_E\\
&<\sum_{j=1}^{\infty}jq_j\left(\frac{2\log j}{j}+\ent_{\topo}(X_0^+)+\frac{c+1}{j}-\frac{2(c+2)}{j}\Big(\frac12+\log j\Big)\right)\\
&=\sum_{j=1}^{\infty}jq_j\left( \ent_{\topo}(X_0^+)-\frac{2(c+1)\log j+1}{j}\right)\\
&<\sum_{j=1}^{\infty}jq_j\, \ent_{\topo}(X_0^+)=\frac{\ent_{\topo}(X_0^+)}{{\mu}(E)},
\end{align*}
where equality comes from Kac's Theorem. 
\end{proof}

\begin{rem} 
These methods do not carry over directly to more general one-sided shifts.
The reason is that in the case where $X_0^+$ is a one-sided shift of finite type, if $x\in X^+\setminus X_0^+$, this is  
a result of localized defects where $x_jx_{j+1}\not\in \cL(X_0^+)$ for some $j$. In particular, if $\dist(x,X_0^+)=2^{-n}$ for some $n\ge1$, then 
$\dist(\sigma x,X_0^+)=2^{-(n-1)}$. However, for a general one-sided subshift $\dist(x,X_0^+)=2^{-n}$
implies $x_0\cdots x_{n-1}\not\in\cL(X_0^+)$, but this does not imply that $x_1\cdots x_{n-1}\not\in\cL(X_0^+)$.
\end{rem}

For a general subshift $X_0^+$ of a one-sided shift $X^+$ we have the following result.

\begin{thm}[One-sided subshifts]\label{thm:super-pins}
Suppose $X^+=\cA^{\N_0}$ is a one-sided full shift and $X_0^+$ is a proper subshift of $X^+$. 
Then the potential $\phi^+:X^+\to\R$ defined by (\ref{eq:def_phi}) has a freezing phase transition with $\ZTA(\phi)$ consisting of measures of maximal entropy of $X_0^+$ provided the decreasing sequence $(a_j)$ satisfies
\[
a_{3j}\ge \frac{2\log j}{j}+\frac{1}{j}\sum_{i=0}^{\lfloor \log_2 j\rfloor}2^i\kappa_{i} +\frac{1}{j},
\] 
where $\kappa_i=\frac{\log |\cL_{2^i}(X_0^+)|}{2^i}-\ent_{\topo}(X_0^+)$.
\end{thm}
\begin{rem}
    We note that the sequence which bounds $a_{3j}$ from below tends to zero as $j\to\infty$. To verify this we use the fact that $\kappa_i$ monotonically decreases to zero and make a crude estimate on the summation above replacing $\kappa_i$ by $\log |\cA|$ in the first half of the terms and by $\kappa_{\lfloor \log_2\sqrt{j}\rfloor}$ in the second half of the terms. We obtain
    \begin{align*}
        \frac{1}{j}\sum_{i=0}^{\lfloor \log_2 j\rfloor}2^i\kappa_{i}&\le \frac{1}{j}\left(2^{\log_2\sqrt{j}}\log|\cA|+\kappa_{\lfloor \log_2\sqrt{j}\rfloor}2^{1+\log_2 j}\right)\\
        & \le \frac{\log|\cA|}{\sqrt{j}}+2\left(\frac{\log |\cL_{\lfloor\sqrt j\rfloor}(X_0^+)|}{\lfloor \sqrt{j}\rfloor}-\ent_{\topo}(X_0^+)\right)\underset{j\to \infty}\longrightarrow 0.
    \end{align*}
\end{rem}
\begin{proof}
  Consider $X=\cA^{\Z}$ and let $X_0\subset X$ be the two-sided subshift generated by the language of $X^+_0$. We define the natural projection map 
\begin{equation}\label{eq:def_projection}
    p:X\to X^+;\quad p((x_i)_{i\in\Z})=(x_i)_{i\in \N}.
\end{equation}
Then $p(X_0)=X_0^+$.  Any shift-invariant measure $\mu^+$ on  $X^+$ can be written as $\mu^+=p_*\mu$, where $\mu$ is a shift-invariant probability measure on $X$. In this case we have $\ent(\mu)=\ent(\mu^+)$. For a given continuous one-sided potential $\phi^+:X^+\to \R$, the corresponding two-sided potential $\phi:X\to \R$ is given by $\phi=\phi^+\circ p$. Clearly, $\int\phi^+\dd\mu^+=\int\phi\dd\mu$, and hence the pressure functions of $\phi$ and $\phi^+$ coincide. We show that $\ent(\mu)+\int\phi\dd \mu< 0$ for any ergodic measure $\mu$ which is not supported on $X_0$.

We consider the product space $X\times Z$ with map $\bar{\sigma}(x,z)=(\sigma x,\sigma z)$ where $Z=\{0,1\}^\Z$. 
  Let $\bar{X}$ be the subshift of $X\times Z$ consisting of pairs $(x,z)$ with the following properties:
\begin{itemize}
\item If $i<j$, and $z_k=0$ for $i<k\le j$ (possibly with $z_i=1$) then $x_i\cdots x_j$ is admissible in $X_0$;
\item If $i<j$, $z_i=1$, $z_j=1$ then $j-i=2^k$ for some $k\in\N_0$ and $x_i\cdots x_{2j-i-1}$ is not admissible in $X_0$.
\end{itemize}
Hence, we partition each $x\in X$ into maximal dyadic blocks which are admissible in $X_0$.  We note that the map $\pi\colon \bar X\to X$ defined by $\pi(x,z)=x$ is a countable-to-one factor map: for any $x$ and for any $n$ such that $z_n=1$, $(z_m)_{m\ge n}$ is uniquely determined by the condition that $(x,z)\in \bar X$. As a consequence, $\ent(\bar\mu)=\ent(\mu)$. 

We refer to the space $\bar{X}$ as the pin-sequence space and to the 1's as pins.We say that a particular pin is a superpin if the dyadic interval which directly follows the pin is not larger than the preceding one. We induce on the set of points $(x,z)$ such that $z_0$ is a superpin, i.e.
\[
E=\{(x,z)\in \bar X\colon z_0=1\text{ and }
\min\{i>0\colon z_i=1\}\le \min\{i>0\colon z_{-i}=1\}\}.
\]

As in the proof of Theorem \ref{thm:one_sided_SFT}, we obtain the induced system $(E,\bar{\mu}_E, \bar{\sigma}_E)$ and the induced potential $\bar{\phi}_E$. We denote the first return time partition of $E$ by $\cQ=\{Q_j\}$ where $Q_j=\{(x,z)\in E: r(x,z)=j\}$. Let $q_j=\bar{\mu}_E(Q_j)$ as before.

We estimate $\int\bar{\phi}_E\dd\bar{\mu}_E$ first. Suppose $(x,z)\in Q_j$. Since the superpin is at $z_j$, we know that the admissible dyadic block which starts at position $j+1$ cannot be longer than the previous dyadic block, and, in particular, cannot be longer than $j$. Hence, the distance to the next pin is at most $j$. Recall that the pins partition $x$ into maximal dyadic blocks admissible in $X_0$, so the block of length twice the distance to this pin which starts at $j+1$ cannot be admissible.    We conclude that $x_j\cdots x_{3j}\notin \cL(X_0)$. Therefore, for  $i=0,\ldots,j-1$ we have $\dist(p(\sigma^i x), X_0^+)\ge \frac{1}{2^{3j}}$, where $p$ is the projection map given by (\ref{eq:def_projection}). We see that $\phi(\sigma^i x)\le -a_{3j}$ for  $i=0,\dots,j-1$, whence $\bar{\phi}_E(x,z)=\sum_{i=0}^{j-1}\phi(\sigma^i)(x)\le -ja_{3j}$. We obtain 
\begin{equation}\label{eq:phi+_estimate}
  \int \bar{\phi}_E\dd\bar{\mu}_E\le -\sum_{j=1}^\infty ja_{3j}q_j.
\end{equation}

We estimate $\ent(\bar{\mu})$ by the entropy of the partition $\cP$ which is a subpartition of $\cQ$,
in which each $Q_j$ is partitioned into length $j$ cylinder sets. We write $\Ent(\cP)\le \Ent(\cQ)+\Ent(\cP|\cQ)$ and note that $\Ent(\cQ)=\sum_{j=1}^{\infty}-q_j\log q_j\le \sum_{j=1}^{\infty}2q_j\log j+1$ as was shown in (\ref{eq:entropy_first_return_partition}). It remains to estimate $\Ent(\cP|\cQ)$.


 Since the pins between two adjacent superpins at $z_0$ and $z_j$ are separated by gaps that are powers of 2 of increasing order, they are uniquely determined by the binary expansion of $j$. Suppose there are $m(j)-1$ pins between $z_0$ and $z_j$ with gaps of length $2^{i_1},\ldots,2^{i_{m(j)-1}}$ so that $j= 2^{i_1}+\dots+2^{i_{m(j)}}$. Clearly, $m(j)-1\le\lfloor \log_2 j\rfloor$. 
 Then the number of different cylinder sets of $Q_j$ is at most $n_{i_1}\times\cdots\times n_{i_{m(j)}}$ where $n_i=|\cL_{2^i}(X_0)|$. Denote $\kappa_i=\frac{\log n_i}{2^i}-h_{\rm top}(X_0) $. Then,
 \begin{align*}\label{eq:entropy+_estimate}
    \ent(\bar{\mu}_E) &  \le \sum_{j=1}^{\infty} 2q_j\log j+1+\sum_{j=1}^{\infty}q_j\log (n_{i_1}\times\cdots\times n_{i_{m(j)}})\\
   & \le \sum_{j=1}^{\infty} q_j\Bigg(2\log j+\sum_{k=1}^{m(j)}\log n_{i_k} +1\Bigg) \\
   & \le \sum_{j=1}^{\infty} q_j\Bigg(2\log j+\sum_{k=1}^{m(j)}2^{i_k}\kappa_{i_k} +j\ent_{\topo}(X_0)+1\Bigg)\\
   & \le \sum_{j=1}^{\infty} jq_j\Bigg(\frac{2\log j}{j}+\frac{1}{j}\sum_{i=0}^{\lfloor \log_2 j\rfloor+1}2^i\kappa_i +\frac{1}{j}+\ent_{\topo}(X_0)\Bigg).
 \end{align*}
 Combining the last inequality with (\ref{eq:phi+_estimate}) we obtain
 \[
 \ent(\bar{\mu_E})+\int \bar{\phi}_E\dd\bar{\mu}_E\le \sum_{j=1}^{\infty} jq_j\Bigg(\frac{2\log j}{j}+\frac{1}{j}\sum_{i=0}^{\lfloor \log_2 j\rfloor}2^i\kappa_{i} +\frac{1}{j}+\ent_{\topo}(X_0)-a_{3j}\Bigg).
 \]
Taking $(a_j)$ to be a monotone sequence such that 
\[
a_{3j}\ge\frac{2\log j}{j}+
\frac{1}{j}\sum_{i=0}^{\lfloor \log_2 j\rfloor}2^i\kappa_{i} +\frac{1}{j}
\] 
gives $\ent(\bar{\mu_E})+\int \bar{\phi}_E\dd\bar{\mu}_E\le \ent_{\topo}(X_0)\sum_{j=1}^{\infty} jq_j=\frac{\ent_{\topo}(X_0)}{\bar{\mu}_E(E)}$, so that $\int\phi\dd\mu+\ent(\mu)\le \ent_{\topo}(X_0)$. 
\end{proof}

\begin{corollary}\label{cor:Thue-Morse}
    Suppose a subshift $X_0^+\subset\cA^\N$ has a word complexity function that grows at most polynomially and is the smallest subshift which supports all its invariant measures.
  Then the potential $\phi$ defined by (\ref{eq:def_phi}) 
  freezes on $X_0^+$ provided $(a_j)$ is a decreasing sequence satisfying $a_j\ge\frac{\log^2 j}{j}$, $j\in\N$.
\end{corollary}
\begin{proof}
    Suppose $|\cL_j(X_0)|\le p(j)$ for all $j$ where $p(j)$ is a polynomial of degree $q\in\N$. Then $\ent(X_0^+)=0$ and $2^i\kappa_i\le cq\log i$ for some constant $c>0$. We then estimate
    \[
    \frac{1}{j}\sum_{i=0}^{\lfloor \log_2 j\rfloor+1}2^i\kappa_{i}\le \frac{cq\log j(\lfloor\log_2 j\rfloor+2)}{j}\le \frac{2cq\log^2 j}{j}.
    \]
\end{proof}
Bruin and Leplaideur \cite{BL1} consider the Thue-Morse one-sided subshift $X_0^+$ and
$\phi(x)=-\frac{1}{j^a}+o\big(\frac{1}{j^a}\big)$ 
if $\dist(x,X_0^+)=2^{-j}$. (For this subshift $\cA=\{0,1\}$ and the generating substitution is given by the rules $0\to 01$, $1\to 10$.) They assert that $\phi$ exhibits a freezing phase transition for $a \in (0,1)$ but undergoes no phase transition for $a > 1$. They further posed the question of whether $\phi$ undergoes a freezing phase transition at $a = 1$. The proofs presented in \cite{BL1} were incomplete but were later corrected in \cite{BCHL}. Additionally, the question of freezing for $\phi(x) = -\frac{1}{j}$ was affirmatively resolved.
We observe that the Thue-Morse subshift has linear complexity. Consequently, Corollary \ref{cor:Thue-Morse} ensures freezing when $\phi(x) \approx -\frac{\log^2 j}{j}$, provided that $\dist(x, X_0) = 2^{-j}$. It also follows from Theorem \ref{nogo-thm} that freezing phase transition does not occur for $\phi\approx -\frac{1}{n^a}$, $a>1$. 

\section{Proofs of some Folklore results}\label{appendix-folklore}

For the reader's convenience, we provide proofs of the following facts mentioned in the introduction and Section \ref{sec:main-results}, as we could not find a readily available reference.
\begin{rem}
It can be readily checked that Theorem \ref{thm:folklore} and Proposition \ref{prop:affine-parts-pressure} below are valid for any general expansive continuous $\Zd$-action on a compact metric space.
\end{rem}

\subsection{Properties of accumulation points of \texorpdfstring{$(\mu_{\beta_n\phi})_n$}{}
as \texorpdfstring{$\beta_n\to\infty$}{beta}}\label{app:folklore-on-ground-states}

Let $\phi:\cA^{\Z^d}\to\R$ be a continuous function.
Recall that a shift-invariant measure $\mu$ is maximizing for $\phi$ if
$\int \phi\dd\mu\geq \int \phi \dd\eta$ for all shift-invariant
measures $\eta$. By weak$^*$-compactness of the set of shift invariant measures and continuity of the mapping $\eta\mapsto \int\phi\dd\eta$,
the maximum is attained, so that $\MM(\phi)$, the set of maximizing measures, is non-empty.
Let
\[
s_\phi=\sup_\nu \sint \phi \dd\nu\quad\mathrm{and}
\quad \ent_\phi=\sup\big\{\ent(\eta): \eta\in \MM(\phi)\big\}.
\]
In statistical physics, $\ent_\phi$, known as the residual entropy, can also be determined as the exponential growth rate of the number of ground configurations in increasingly large boxes. \cite{AL}.

Recall also that the entropy function $\eta\mapsto \ent(\eta)$, where $\eta$ runs over the set of shift-invariant measures, is 
bounded between $0$ and $\log|\cA|$, and is upper semi-continuous \cite[Theorem 4.5.6]{Keller}.

Using the variational principle, we define the pressure function $p_\phi:\R_+\to\R$ as
\[
\beta\mapsto p_\phi(\beta):=\sup_{\nu}\Big\{\ent(\nu)+\beta \sint \phi\dd\nu\Big\},
\]
where the supremum is taken over all shift-invariant measures. Since the functional $\nu\mapsto \ent(\nu)+\beta\int\phi\dd\nu$
is upper semi-continuous, the supremum is attained again. The measures 
for which the supremum is attained are the equilibrium states for $\beta\phi$.
As a convex function, $p_\phi$ has left and right derivatives $p'_{\phi,-}(\beta)$ and
$p'_{\phi,+}(\beta)$ at each $\beta>0$, and
\begin{equation}\label{leftrightderivative}
p'_{\phi,-}(\beta)\leq \int \phi\dd\mu_{\beta\phi}\leq p'_{\phi,-}(\beta)\text{ for any $\mu_{\beta\phi}\in\ES(\beta\phi)$}.
\end{equation}
(When $p'_{\phi,-}(\beta)=p'_{\phi,+}(\beta)$, that is, when the pressure function is differentiable at $\beta$, we thus have
$p'_\phi(\beta)=\int \phi \dd\mu_{\beta\phi}$.)
These facts follow from a combination of \cite[Chapter 16]{Georgii} and \cite[Theorems 3.4 and 3.12]{Ruelle}.

\begin{thm}\label{thm:folklore}
Let $\phi:\cA^{\Z^d}\to\R$ be a continuous function. We have the following facts:
\begin{itemize}
\item[\textup{(i)}] 
$\ZTA(\phi)\subseteq \MM(\phi)$, and if $\mu\in\ZTA(\phi)$ then 
$\ent(\mu)=\sup\{\ent(\nu):\nu\in\MM(\phi)\}$.
\item[\textup{(ii)}] The pressure function has a slant asymptote. More precisely,
\[
\lim_{\beta\to+\infty} \big(\, p_\phi(\beta)-(s_\phi\beta +\ent_\phi)\big)=0. 
\]
\end{itemize}
\end{thm}
\begin{proof}
Firstly, if $\eta\in\MM(\phi)$, the variational principle yields the bounds
\begin{equation}\label{eq:VP-est}
    h(\eta)+\beta s_\phi\le p(\beta\phi)\le \log|\cA|+\beta s_\phi.
\end{equation}
Now let $\mu^*\in \ZTA(\phi)$ and let $\beta_n\to\infty$ and $(\mu_n)$ be such that
$\mu_n$ is an equilibrium state for $\beta_n\phi$ and $\mu_n\leadsto\mu^*$.
By \eqref{eq:VP-est}, $\int \beta_n\phi\dd\mu_n+\ent(\mu_n)\geq \beta_ns_\phi+\ent(\eta)$ 
for all $n$. Dividing by $\beta_n$ and taking the limit, we obtain $\int\phi\dd\mu^*\ge s_\phi$,
so that $\mu^*\in\MM(\phi)$. 
That is, we have shown $\ZTA(\phi)\subseteq\MM(\phi)$. 

Now we prove that, if $\mu^*\in\ZTA(\phi)$ then it has the maximal entropy over the set of
maximizing measures for $\phi$. Along the way, we will
also prove that $\lim_{\beta\to+\infty} \big(\, p_\phi(\beta)-s_\phi\beta\big)= \ent_\phi$, which will complete the proof of statement (ii).
Define the auxiliary function $g_\phi(\beta):=p_\phi(\beta)-s_\phi \beta$ (for $\beta>0$). From
\eqref{eq:VP-est}, we see 
\begin{equation}\label{eq-gh}
g_\phi(\beta)\geq \ent(\eta),\;\forall \beta>0.    
\end{equation}
We can compute the left derivative of $g_\phi$ to show that it is a decreasing function. Indeed,
\[
g'_{\phi,-}(\beta)=p'_{\phi,-}-s_\phi\leq \int \phi \dd\mu_{\beta\phi}-s_\phi\leq 0.
\]
Thus, the limit of $g_\phi(\beta)$ as $\beta\to\infty$ exists, so using \eqref{eq-gh}, we see
\begin{equation}\label{unbout}
\lim_{\beta\to\infty} g_\phi(\beta)\geq \ent(\eta).  
\end{equation}
Taking a supremum of \eqref{unbout} over $\eta\in\MM(\phi)$, we see
\begin{equation}\label{unbout2}
\lim_{\beta\to\infty} g_\phi(\beta)\geq \sup\{\ent(\eta)\colon\eta\in\MM(\phi)\}.
\end{equation}
Let $\alpha>0$ and let $\mu_{\alpha\phi}$ be any equilibrium state for $\alpha\phi$,
We see from the variational principle and the definition of $s_\phi$ that
\[
g_\phi(\alpha)=\ent(\mu_{\alpha\phi})+\alpha\Big(\,\sint \varphi \dd\mu_{\alpha\phi}-s_\phi\Big)\leq \ent(\mu_{\alpha\phi}),
\]
so that since $g_\phi$ is decreasing,
\[
\ent(\mu_{\alpha\phi})\ge \lim_{\beta\to\infty}g_\phi(\beta)\text{ for all $\alpha$.}
\]

Let $\mu^*\in\ZTA(\phi)$ and let $\beta_n\to\infty$ and $(\mu_n)$ be a sequence of equilibrium states for $\beta_n\phi$
such that $\mu_n\leadsto\mu^*$. 
By upper semi-continuity of the entropy function we deduce that
\begin{equation}\label{unbout3}
\ent(\mu^*)\ge \limsup_n \ent(\mu_{\beta_{n}\phi})\ge \lim_{\beta\to\infty} g_\phi(\beta)
\end{equation}
Since we have shown $\mu^*\in\MM(\phi)$, combining \eqref{unbout2} with \eqref{unbout3}, we see
\begin{equation*}
    \lim_{\beta\to\infty}g_\phi(\beta)\geq \sup\{\ent(\eta)\colon\eta\in\MM(\phi)\}\ge \ent(\mu^*)\ge \lim_{\beta\to\infty}g_\phi(\beta).
\end{equation*}
This finishes the proof of statement (ii) and the second part of statement (i).
\end{proof}

\subsection{When the pressure function is affine on a segment}

The following proposition provides insight into the affine parts of the pressure function graph.
\begin{prop}\label{prop:affine-parts-pressure}
Let $\phi\colon\cA^\Zd\to\R$ be a continuous function. Then, for any $\beta_1<\beta_2$, the following assertions are equivalent:
\begin{enumerate}
    \item $\ES(\beta_1\phi)\cap \ES(\beta_2\phi)\neq \emptyset$;\label{cond:nonempt}
    \item $p_\phi$ is affine on $\left[
    \beta_1,\beta_2\right]$;\label{cond:affine}
    \item $\ES(\beta\phi)=\ES(\beta_1\phi)\cap \ES(\beta_2\phi)$ for all $\beta\in(\beta_1,\beta_2)$.\label{cond:intersect}
\end{enumerate}
\end{prop}
\begin{proof}
Suppose \eqref{cond:nonempt} holds and let $\mu\in\ES(\beta_1\phi)\cap\ES(\beta_2\phi)$ 
and $s\in (0,1)$ so that $(1-s)\beta_1+s\beta_2\in(\beta_1,\beta_2)$.
Then since $p_\phi$ is convex, $p_\phi((1-s)\beta_1+s\beta_2)\le (1-s)p_\phi(\beta_1)+sp_\phi(\beta_2)$.
On the other hand, $p_\phi((1-s)\beta_1+s\beta_2)\ge \ent(\mu)+((1-s)\beta_1+s\beta_2)\int \phi\dd\mu=
(1-s)p_\phi(\beta_1)+sp_\phi(\beta_2)$, so that $p_\phi$ is affine on $[\beta_1,\beta_2]$ as claimed
and $\mu\in ES(((1-s)\beta_1+s\beta_2)\phi)$.

Now suppose that $p_\phi$ is affine on $[\beta_1,\beta_2]$ and let $\beta=(1-s)\beta_1+s\beta_2$
with $0<s<1$ and $\mu\in \ES(\beta\phi)$. Then $p_\phi'(\beta)=\int\phi\dd\mu$, so that 
$\ent(\mu)+t\int\phi\dd\mu=\ent(\mu)+\beta\int\phi\dd\mu+(t-\beta)\int\phi\dd\mu=p_\phi(\beta)+
(t-\beta)p_\phi'(\beta)=p_\phi(t)$ for all $t\in[\beta_1,\beta_2]$. It follows that
$\mu\in \ES(\beta_1\phi)\cap \ES(\beta_2\phi)$. Conversely, if $\mu\in \ES(\beta_1\phi)\cap \ES(\beta_2\phi)$
the first part implies that $\mu\in \ES(\beta\phi)$. 

The implication \eqref{cond:intersect} implies \eqref{cond:nonempt} is trivial.
\end{proof}

\subsection{Proof of Proposition \ref{prop:freezing=slant-asymptote}}
 
By Proposition \ref{prop:affine-parts-pressure},
$p_\phi$ is affine on $[\,\beta_1,\beta_2]$
if and only if $\ES(\beta\phi)$ is equal for all $\beta\in(\beta_1,\beta_2)$. 
If $\phi$ freezes at $\betac$, by 
continuity of $p_\phi$, we see $p_\phi$ is affine on $[\,\betac,\infty)$
but no larger interval. 


For the converse, using Proposition \ref{prop:affine-parts-pressure}, if $p_\phi$ is affine on $[\,\betac,\infty)$, 
then $\ES(\beta\phi)=\ES(\beta'\phi)$ for all $\beta,\beta'>\betac$. Meanwhile, if $\beta<\betac<\beta'$, 
since $p_\phi$ is not affine on $[\,\beta,\beta']$, we see $\ES(\beta\phi)\cap\ES(\beta'\phi)=\emptyset$.
Hence $\phi$ has a freezing phase transition at $\betac$.

By Theorem \ref{thm:folklore}, $p_\phi(\beta)$ approaches the slant asymptote $s_\phi\beta+\ent_\phi$
as $\beta\to\infty$, so if $p_\phi$ is affine
on $[\betac,\infty)$, in fact $p_\phi(\beta)=s_\phi\beta+\ent_\phi$ on that range.

\subsection{Gibbs states have full support}\label{app:Gibbs} 

For completeness, and since we were unable to find a suitable reference, we provide a proof of the following well-known result invoked in Section \ref{sec:proof-theorem-B'}.  

We begin by quickly recalling the definition of a Gibbs state associated with a given interaction $\Phi \in \mathcal{S}$ (see \eqref{SmallBanachSpace}). For a comprehensive treatment, we refer to \cite{Georgii}.  

Let $\mathfrak{B}$ denote the Borel $\sigma$-algebra on $X$, which coincides with the Borel $\sigma$-algebra generated by cylinder sets. Given a subset $\Lambda \subset \mathbb{Z}^d$, we define $\mathfrak{B}_\Lambda$ as the $\sigma$-algebra generated by the coordinate maps $x \mapsto x_\bi$ for $\bi \in \Lambda$.  

The Gibbsian specification for $\Phi$ is given by  
\[
\gamma^\Phi_\Lambda(x|y):=
\frac{\e^{U^\Phi_\Lambda(x|y)}}{Z^\Phi_\Lambda(y)}
\]
where
\[
U^\Phi_\Lambda(x|y):=\sum_{\scaleto{\substack{\Lambda'\Subset\Z^d \\\Lambda'\cap\Lambda\neq \emptyset}}{12pt}} \Phi_{\Lambda'}(x\vert_{\Lambda}y\vert_{\Lambda^\comp}),
\]
and $Z^\Phi_\Lambda(y)$ is the normalization factor, $y\vert_{\Lambda^\comp}$ represents the boundary condition, and $\Lambda \Subset \mathbb{Z}^d$.  
(For each $y$, $\big(\gamma^\Phi_\Lambda(\cdot|y)\big)_{\Lambda\Subset \Z^d}$ is a family of probability
kernels on $(X,\mathfrak{B})$ subject to natural conditions we don't list here.)

A probability measure $\mu$ on $(X,\mathfrak{B})$ is said to be a Gibbs state for $\Phi$ if, for every finite subset $\Lambda \Subset \mathbb{Z}^d$, the conditional probability $\gamma^\Phi_\Lambda(x|\cdot)$ provides a version of  
\[
\mu\big(x\vert_{\Lambda}\big|\mathfrak{B}_{\Lambda^\comp}\big).
\]
Equivalently, this means that for all $B \in \mathfrak{B}$ and all finite $\Lambda \Subset \mathbb{Z}^d$, we have  
\begin{equation}\label{DLR}
\mu(B)=\int \sum_{x\,\in \cA^\Lambda}\gamma^\Phi_\Lambda(x|y)\,\mathds{1}_{B}(x\vert_{\Lambda}
y\vert_{\Lambda^\comp})\dd\mu(y).
\end{equation}
(This is called the DLR equation -- after Dobrushin, Lanford and Ruelle.) It is well known that at least one Gibbs state for $\Phi$ always exists.  

Finally, for $x\in X$ and $n \geq 1$, we define the cylinder set centered at $x$ with support on $\Lambda_n$ as  
\[
C_n(x) := \big\{y \in X : y\vert_{\Lambda_n} = x\vert_{\Lambda_n}\big\}.
\]

\begin{prop}
Let $\Phi\in\mathcal{S}$ and $\mu$ be a Gibbs state for $\Phi$.
Then there exists $\rho=\rho(\Phi,\cA)\in \left(0,1\right)$ such that, for all
$x\in X, n\geq 1$,  
\begin{equation}\label{lbGm}
\mu(C_n(x))\geq \rho^{|\Lambda_n|}.    
\end{equation}
\end{prop}
\begin{proof}
Since $\Phi\in\mathcal{S}$, we have the bound
\[
\big\vert U^\Phi_\Lambda(x|y)\big\vert \leq |\Lambda| \|\Phi\|_{\mathcal{S}},
\]
which is true for all $x,y\in X$, hence
\[
\e^{-|\Lambda| \|\Phi\|_{\mathcal{S}}}\leq \e^{U^\Phi_\Lambda(x|y)} \leq \e^{|\Lambda| \|\Phi\|_{\mathcal{S}}}\quad\text{and}\quad Z^\Phi_\Lambda(y) \leq |\cA|^{|\Lambda|}
\e^{|\Lambda| \|\Phi\|_{\mathcal{S}}}.
\]
Therefore
\[
\gamma^\Phi_\Lambda(x|y) \geq |\cA|^{-|\Lambda|}\e^{-2|\Lambda| \|\Phi\|_{\mathcal{S}}},
\]
which by \eqref{DLR} yields
\[
\mu(C_n(x)) = \int \gamma^\Phi_{\statbox n}(x|y) \dd\mu(y)\geq 
|\cA|^{-|\statbox n|}\e^{-2|\statbox n| \|\Phi\|_{\mathcal{S}}}
\]
since $\mu(X)=1$. Letting $\rho:=|\cA|^{-1}\e^{-2\|\Phi\|_{\mathcal{S}}}$, this concludes the proof. 
\end{proof}

\bigskip
 
\noindent \textbf{Acknowledgement.}
The authors warmly thank Aernout van Enter for his thoughtful email exchanges. We did our best to draw a connection with statistical mechanics, and our conversations with him played a meaningful role in guiding that process.
AQ and TK acknowledge the financial support from CNRS and the hospitality of the Centre de Physique Th\'eorique (CPHT) at \'Ecole Polytechnique, where part of this work was conducted.


\end{document}